\documentclass[12pt,leqno]{amsart}
\usepackage{amsmath,color, multirow}
\usepackage{graphicx}
\newcommand{\R}{{\mathbb R}}

\newcommand{\Var}{{\rm Var}}
\newcommand{\Cov}{{\rm Cov}}
\DeclareMathOperator{\BB}{\textit{BB}}
\newcommand{\claw}{{\stackrel{\mathcal{D}}{\longrightarrow}}}
\newtheorem{lemma}{Lemma}[section]

\newtheorem{theorem}[lemma]{Theorem}

\newtheorem{corollary}[lemma]{Corollary}
\newtheorem{example}[lemma]{Example}

\newtheorem{remark}[lemma]{Remark}

\begin{document}
\title[Power of Change-Point Tests for LRD Data]
{Power of Change-Point Tests for Long-Range Dependent Data}
\author[H. Dehling]{Herold Dehling}
\author[A. Rooch]{Aeneas Rooch}
\author[M.~S. Taqqu]{Murad S.~Taqqu}
\today

\address{
Fakult\"at f\"ur Mathematik, Ruhr-Universit\"at Bo\-chum, 
44780 Bochum, Germany}
\email{herold.dehling@rub.de}
\address{Fakult\"at f\"ur Mathematik, Ruhr-Universit\"at Bo\-chum, 
44780 Bochum, Germany}
\email{aeneas.rooch@rub.de}
\address{Department of Mathematics, 111 Cummington St.,  
Boston University,  Boston MA 02215
}
\email{murad@bu.edu}

\thanks{Herold Dehling and Aeneas Rooch were supported in part by the 
Collaborative Research Grant 823, Project C3 {\em Analysis of Structural 
Change in Dynamic Processes}, of the German Research 
Foundation. Murad S.~Taqqu was supported in part by NSF grant DMS-1007616 at Boston University}
\begin{abstract}
We investigate the power of the CUSUM test and the Wilcoxon change-point tests for a shift in the mean 
of a process with long-range dependent noise. We derive analytic formulas for the power of these tests
under local alternatives. These results enable  us to calculate the asymptotic relative efficiency (ARE) of the 
CUSUM test and the Wilcoxon change point test. We obtain the surprising result that for Gaussian data, the ARE of these two tests equals $1$, in contrast to the case of i.i.d. noise when the ARE is known to be $3/\pi$.

\end{abstract}
\maketitle
\tableofcontents
\section{Introduction}
Statistical tests for the presence of changes in the structure of time series are of great importance in a 
wide range of scientific discussions, e.g.\ regarding economic, technological and climate data. Many procedures for detecting changes and for estimating change-points have been proposed in the literature; see e.g.\ Cs\"org\H{o} and Horvath (1997) for a detailed exposition. In the case of independent data, the theory is quite satisfactory. For various types of change-point models, statistical procedures have been proposed and
their properties investigated. In contrast, the situation is different for dependent data, such as encountered in time series models. For dependent data, most research has focused on linear procedures, such as cumulative sum (CUSUM) tests, and there are many open problems when it comes to other types of test procedures, e.g. those used in robust statistics.

In the present paper, we study the change-point problem for long-range dependent data. Specifically, we will test the hypothesis that the process is stationary against the alternative that there is a change in the mean. The classical test statistic for this problem is the CUSUM statistic, 
\begin{equation}
\max_{1\leq k\leq n-1} \left|\sum_{i=1}^k X_i -\frac{k}{n} \sum_{i=1}^n X_i\right|.
\label{eq:cusum}
\end{equation} 
When the test statistic is large, one infers that there is a change in the mean.
The CUSUM test has good properties when the underlying process is Gaussian.
The asymptotic distribution of the CUSUM test in the presence of long-range dependent data has been investigated by Horv\'ath and Kokoszka (1997).

However, the CUSUM test is not robust against possible outliers in the data, because the sum $\sum_{i=1}^k X_i$ can change drastically when there are outliers. Recently, Dehling, Rooch and Taqqu (2013) have proposed a robust alternative to the CUSUM test, which is based on the Wilcoxon two-sample rank statistic. The corresponding "Wilcoxon change-point test"
uses the test statistic 
\begin{equation}
\max_{1\leq k\leq n-1}  \left| \sum_{i=1}^k \sum_{j=k+1}^n
\left(1_{\{ X_i \leq X_j\}} -\frac{1}{2}\right) \right|.
\label{eq:wilcoxon}
\end{equation}
One rejects the null hypothesis when this test statistic is large. Rank tests for change-point problems have been studied earlier by Antoch et al (2008),  in the presence of i.i.d. data, and by Wang (2008) for linear processes. 

In their paper, Dehling, Rooch and Taqqu (2013) investigated the asymptotic  distribution of the Wilcoxon change-point test under the null hypothesis of
no change, in the presence of long-range dependence. Moreover, they performed a simulation study to compare the finite sample performance and the power of the CUSUM test based on (\ref{eq:cusum}) and the Wilcoxon change-point test based on (\ref{eq:wilcoxon}).\footnote{Dehling et al. (2013) called the CUSUM test, the "difference of means test", and called the
Wilcoxon change-point test, the "Wilcoxon-type" test.}

In the present paper, we study the power of the CUSUM test and the Wilcoxon change-point test for a shift in the mean of a long-range dependent process. We will calculate the power under local alternatives, where the
height of the shift decreases with the sample size $n$ in such a way that the tests have non-trivial limit power as $n\rightarrow \infty$. These results enable us to compute the asymptotic relative efficiency (ARE) of
the CUSUM and the Wilcoxon change-point tests, which is defined as the limit of the ratio of the sample sizes required to obtain a given power. We obtain the surprising result that the ARE of these two tests equals $1$ in the case of long-range dependent Gaussian data. This is in contrast with the case of i.i.d.\ and short-range dependent data, where the ARE of the Wilcoxon change-point test with respect to the CUSUM test is $3/\pi$.
In the context of M-estimation of a location parameter, a similar phenomenon has been observed by Beran (1991); see also Beran (1994), Corollary~8.1.

We consider a model where the observations are
generated by a stochastic process $(X_i)_{i\geq 1}$ of the type
\begin{equation}
 X_i= \mu_i + \epsilon_i, 
\label{eq:Xeps}
\end{equation}
where $(\epsilon_i)_{i\geq 1}$ is a long-range dependent stationary process 
with 
mean zero, finite variance and where $(\mu_i)_{i\geq 1}$ are the unknown means.
We focus on the case when $(\epsilon_i)_{i\geq 1}$
is an instantaneous functional of a stationary Gaussian process $(\xi_i)_{i\geq 1}$ with 
non-summable covariances,  i.e.
\[
 \epsilon_i=G(\xi_i), \quad i\geq 1.
\]
We assume that $(\xi_i)_{i\geq 1}$ is a long-range dependent (LRD), mean-zero Gaussian process with variance
$E(\xi_i^2)=1$  and  autocovariance function 
\begin{equation}
 \rho(k)=k^{-D}L(k), \quad k\geq 1,
\label{eq:acf}
\end{equation}
where $0<D<1$, and where $L(k)$ is a slowly varying function. 
Moreover,
$G:\R\rightarrow \R$ is a measurable function satisfying $E(G(\xi_i))=0$. 

Based on observations $X_1,\ldots, X_n$, we wish to test the hypothesis
\[
 H:\, \mu_1=\ldots=\mu_n
\] 
that there is no change in the means of the data against the alternative
\begin{equation}
 A:\, \mu_1=\ldots=\mu_k\neq \mu_{k+1}=\ldots=\mu_n,
 \mbox{ for some } k\in \{1,\ldots,n-1\}.
\label{eq:A-alt}
\end{equation}
We shall refer to this test problem as $(H,A)$.

Dehling, Rooch and Taqqu (2013) have studied two tests for this  change-point problem,
namely the  Wilcoxon change-point test which is based on the test statistic
\begin{equation}
 W_n=\frac{1}{n\, d_n}\max_{1\leq k\leq n-1} 
 \left|\sum_{i=1}^k \sum_{j=k+1}^n \left(1_{\{X_i\leq X_j\}} - \frac{1}{2}  
 \right)\right|.
\label{eq:wil}
\end{equation}
and the CUSUM  test which uses the test statistic
\begin{equation}  
 D_n :=\frac{1}{n\, d_n}  \max_{1 \leq k \leq n-1} \left| 
 \sum_{i=1}^{k} 
 \sum_{j=k+1}^n  \left(X_j - X_i\right)
\right|.
\label{eq:dif}
\end{equation}
Observe that the normalization $d_n$, which will be specified below,  is the same for both tests.
These tests are similar in spirit. They compare the first part of the sample to the second part. The 
Wilcoxon change-point test (\ref{eq:wil}) involves the rank of the data whereas the CUSUM test (\ref{eq:dif}) involves their values. One rejects the null  hypothesis of no change when these test statistics are large.

Dehling, Rooch and Taqqu (2013) investigated the asymptotic distribution of these test statistics under the null hypothesis $H$ of no change in the means. In addition, they calculated the power of these tests numerically via a Monte-Carlo simulation. In this paper, we will compute the power of the above test statistics under a local alternative. More specifically, we shall consider the following sequence of alternatives
\begin{equation}
 A_{\tau,h_n}(n):\; \mu_i
 =\left\{ \begin{array}{ll}
  \mu & \mbox{ for } i=1,\ldots,[n\tau] \\
  \mu+h_n &  \mbox{ for } i= [n\tau] +1,\ldots,n,
 \end{array}
\right.
\label{eq:local-alt}
\end{equation}
where $0\leq \tau \leq 1$. Observe that the level shift $h_n$ depends on the sample size $n$.

\section{Power of the CUSUM Test under Local Alternatives}
We will first investigate the asymptotic distribution of the process
\begin{equation}
 D_n(\lambda):=\frac{1}{n\, d_n} \sum_{i=1}^{[n\lambda]} \sum_{j=[n\lambda]+1}^n
 (X_j-X_i), \; 0\leq \lambda\leq 1.
\label{eq:dn-lambda}
\end{equation}
To do so, we consider the Hermite expansion of $G(\xi_i)$, namely
\[
 G(\xi_i)=\sum_{q=1}^\infty \frac{a_q}{q!} H_q(\xi_i),
\]
where $H_q$ is the $q$-th order Hermite polynomial. We define the Hermite rank of the function $G$ as 
\[
  m=\min\{q\geq 1\,:\, a_q\neq 0 \}
\]
and introduce the normalization constants
\begin{equation}
d_n^2=\Var\left(\sum_{j=1}^n H_m(\xi_i)\right).
\label{eq:dn1}
\end{equation}
We suppose $0<D<\frac{1}{m}$, in which case 
\begin{equation}
d_n^2\sim \kappa_m\, n^{2-mD}\, L^m(n),
\label{eq:dn2}
\end{equation}
where $\kappa_m=2(m!)/(1-Dm)(2-Dm)$. Here we use the symbol $a_n\sim b_n$ to denote 
$a_n/b_n\rightarrow 0$ as $n\rightarrow \infty$.

Under the null hypothesis $H$ of no level shift, we get that the process 
$(D_n(\lambda))_{0\leq \lambda\leq 1}$ in (\ref{eq:dn-lambda})
converges in distribution towards the process
\begin{equation}
  \frac{a_m}{m!}(\lambda Z_m(1) -  Z_m(\lambda))_{0\leq \lambda\leq 1},
\label{eq:lim-1}
\end{equation}
where $m$ denotes the Hermite rank of $G$ and where 
\begin{equation}
 a_m=E(H_m(\xi)G(\xi));
\label{eq:a-m}
\end{equation}
see Dehling, Rooch and Taqqu (2013), proof of Theorem~3. The process $(Z_m(\lambda))_{\lambda \geq 0}$ denotes the
$m$-th order Hermite process with Hurst parameter $H=1-\frac{Dm}{2}\in (\frac{1}{2},1)$. It is Gaussian
(namely fractional Brownian motion) when $m=1$, but it is non-Gaussian when $m\geq 2$. For various
representations of the Hermite process $(Z_m(\lambda))_{\lambda \geq 0}$, see Pipiras and Taqqu (2010).

In view of (\ref{eq:Xeps}), under the alternative $A$ in (\ref{eq:A-alt}), we need to consider
\begin{equation}
  D_n(\lambda)=\frac{1}{n\, d_n} \sum_{i=1}^{[n\lambda]} \sum_{j=[n\lambda]+1}^n
 (G(\xi_j)-G(\xi_i)) + \frac{1}{n\, d_n} \sum_{i=1}^{[n\lambda]} \sum_{j=[n\lambda]+1}^n
 (\mu_j-\mu_i).
\label{eq:dn-alt}
\end{equation}
Observe that the statistic $D_n(\lambda)$ presumes that the jump occurs at time $[n\lambda]+1$, whereas the local alternative $A_{\tau,h_n}(n)$ involves a jump at $[n\tau]+1$. There will therefore be an interplay between $\lambda$ and $\tau$. In fact, under the local alternative $A_{\tau,h_n}(n)$ in (\ref{eq:local-alt}), we get 
\begin{equation}
  \frac{1}{n\, d_n} \sum_{i=1}^{[n\lambda]} \sum_{j=[n\lambda]+1}^n
 (\mu_j-\mu_i)
 =\left\{
 \begin{array}{ll}
   \frac{h_n}{n\, d_n} [\lambda n](n-[\tau n]) & \mbox{ for } \lambda \leq \tau\\
  \frac{h_n}{n\, d_n}(n-[\lambda n])[\tau n] & \mbox{ for } \lambda \geq \tau.
 \end{array}
 \right.
\label{eq:exp-alt}
\end{equation}
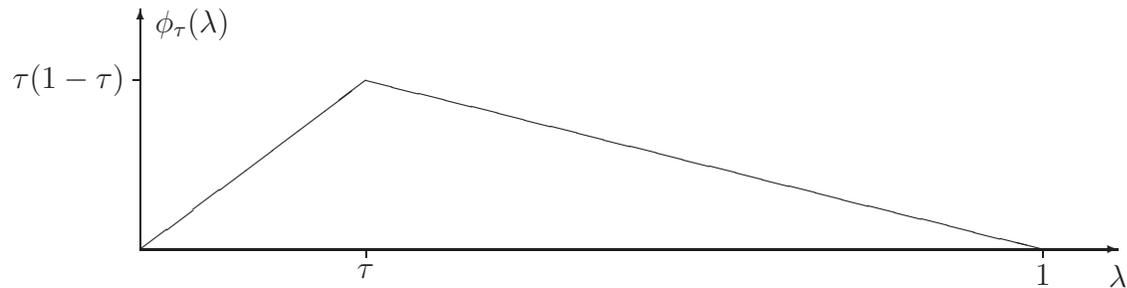
\begin{figure}[h]
\setlength{\unitlength}{1mm}
\begin{picture}(130,35)(0,0)
\put(0,0){\vector(1,0){130}}
\put(0,0){\vector(0,1){32}}
\put(0,0){\line(4,3){30}}
\put(30,22.5){\line(4,-1){90}}
\put(30,0){\line(0,-1){1}}
\put(-2,22.5){\makebox(0,0)[r]{$\tau(1-\tau)$}}
\put(0,22.5){\line(-1,0){1}}
\put(30,-2){\makebox(0,0)[t]{$\tau$}}
\put(120,0){\line(0,-1){1}}
\put(120,-2){\makebox(0,0)[t]{$1$}}
\put(130,-2){\makebox(0,0)[t]{$\lambda$}}
\put(2,30){\makebox(0,0)[l]{$\phi_\tau(\lambda)$}}
\end{picture}
\caption{Graph of the function $\phi_\tau$; see (\ref{eq:phi-tau}).}
\label{fig:phi-tau}
\end{figure}

We introduce the function $\phi_\tau:[0,1] \rightarrow \R$ by
\begin{equation}
 \phi_\tau(\lambda)=\left\{
 \begin{array}{ll}
  \lambda(1-\tau) & \mbox{ for } \lambda \leq \tau\\
  (1-\lambda)\tau & \mbox{ for } \lambda \geq \tau,
 \end{array}
 \right.
\label{eq:phi-tau}
\end{equation}
which takes its maximum value $\tau(1-\tau)$ at $\lambda=\tau$; see Figure~\ref{fig:phi-tau}.
Note that for large $n$, we get
\begin{equation}
 \frac{1}{n\, d_n} \sum_{i=1}^{[n\lambda]} \sum_{j=[n\lambda]+1}^n
 (\mu_j-\mu_i) \sim \frac{n\, h_n}{d_n} \, \phi_\tau(\lambda).
\label{eq:mu-diff}
\end{equation}
Thus, in order for the second term in (\ref{eq:dn-alt}) to converge as $n\rightarrow\infty$, we have to choose the level shift $h_n\sim c\, d_n/n$. When $n$ is large, this is exactly the order of the level shift that can be detected with a nontrivial power, that is with a power which is neither $0$ nor $1$.

\begin{theorem}
Let $(\xi_i)_{i\geq 1}$ be a stationary Gaussian process with mean zero, variance $1$ and autocovariance function as in (\ref{eq:acf}) with $0<D<\frac{1}{m}$. Moreover, let $G:\R\rightarrow \R$ be a measurable function satisfying $E(G^2(\xi))<\infty$ and define $X_i=\mu_i+G(\xi_i)$. Then under 
 the local alternative $A_{\tau,h_n}(n)$ with 
\begin{equation}
 h_n\sim \frac{d_n}{n}\, c,
\label{eq:hn}
\end{equation}
for an arbitrary constant $c$, the process  
$(D_n(\lambda))_{0\leq \lambda \leq 1}$ in (\ref{eq:dn-alt}) converges in distribution to the process
\begin{equation}
 \left( \frac{a_m}{m!}\left(\lambda Z_m(1)-Z_m(\lambda)\right)+c\, \phi_\tau (\lambda)\right)_{0\leq \lambda \leq 1},
\label{eq:dif-alt}
\end{equation}
where $(Z_m(\lambda))_{\lambda \geq 0}$ denotes the $m$-th order Hermite process with Hurst parameter
$H=1-\frac{Dm}{2}\in (\frac{1}{2},1)$, where $a_m$ is given by (\ref{eq:a-m}) and $\phi_\tau(\lambda)$ by
(\ref{eq:phi-tau}).
\label{th:dif-alt}
\end{theorem}
{\em Proof.}  We use the decomposition (\ref{eq:dn-alt}).  The first term on the right hand side has the same distribution as $D_n(\lambda)$ under the hypothesis, and thus converges in distribution to $\frac{a_m}{m!} (\lambda Z_m(1)-Z_m(\lambda))$.  Regarding the second term, we observe that by (\ref{eq:hn}) and  (\ref{eq:exp-alt}) we get
\begin{eqnarray*}
   \frac{1}{n\, d_n} \sum_{i=1}^{[n\lambda]} \sum_{j=[n\lambda]+1}^n
  (\mu_j-\mu_i)& \sim &\left\{
 \begin{array}{ll}
   \frac{c}{n^2} [\lambda n](n-[\tau n]) & \mbox{ for } \lambda \leq \tau\\
     \frac{c}{n^2}(n-[\lambda n])[\tau n] & \mbox{ for } \lambda \geq \tau.
 \end{array}
 \right. \\
 &\rightarrow &c\,\phi_\tau(\lambda),
\end{eqnarray*}
uniformly in $\lambda \in [0,1]$, as $n\rightarrow \infty$.
\hfill $\Box$
\begin{remark}{\rm
(i) Observe that for $c=0$ we recover the limit distribution under the null hypothesis. Thus, Theorem~\ref{th:dif-alt} is a generalization of the results obtained previously under the null hypothesis. The limit process is a fractional bridge process. When $m=1$, this process is a fractional Gaussian bridge. For $m>1$, the process is non-Gaussian.
\\[1mm]
(ii) Under the local alternative, i.e. when $c\neq 0$, the limit process is the sum of a fractional bridge process 
and the deterministic function $c\, \phi_\tau$.  
}
\end{remark}

As an application of the continuous mapping theorem, we obtain the following corollary.
\begin{corollary} \label{cor:Dn}
Under the alternative $A_{\tau,h_n}(n)$ with $h_n\sim \frac{d_n}{n}c$, $D_n$ as defined in 
(\ref{eq:dif}) converges in distribution to 
\begin{equation}
   \sup_{0\leq \lambda\leq 1}
  \left|\frac{a_m}{m!} \left(\lambda Z_m(1)-Z_m(\lambda)\right)+c\, \phi_\tau (\lambda) \right|.
\label{eq:sup1}
\end{equation}
\end{corollary}

\begin{remark}{\rm
(i) The limit distribution (\ref{eq:sup1}) depends on the constant $c$. For $c=0$, we obtain the limit distribution under the null hypthesis. Quantiles of this limit distribution were calculated numerically via a Monte-Carlo simulation by  Dehling, Rooch and Taqqu (2013), Table~1.  Increasing the value of $|c|$ leads to a shift of the distribution to the right. If $c=\infty$,
that is, if $h_n$ tends slower to zero than $\frac{d_n}{n}c$ for any $c>0$, then the correct normalization
for $D_n(\lambda)$ should go to $\infty$ at a higher rate which would kill the random part
$(\lambda Z_m(1)-Z_m(\lambda))$ in (\ref{eq:dif-alt}), and hence the level shift could be detected precisely. The power of the asymptotic test would be equal to $1$ in this case. 
\\[1mm]
(ii) For a given $\tau\in [0,1]$, the function $\phi_\tau(\lambda)$ takes its maximum value in $\lambda=\tau$, and this maximum value equals $\tau(1-\tau)$. Thus, for values of $\tau$ close to $0$ and close to $1$, $\tau(1-\tau)$ is close to $0$, and thus the effect of adding the term $c\phi_\tau(\lambda)$ is rather small. As a result, the power of the test is small at  level shifts that occur  very early or very late in the process. 
\\[1mm]
(iii) The higher the level shift and the closer $\lambda$ is to $\tau$, the easier it is to detect the level shift.
\\[1mm]
(iv) If the observations are short-range dependent, one can typically detect level shifts $h_n$ of size 
$\frac{\sqrt{n}}{n}=\frac{1}{\sqrt{n}}$, but here, because of long-range dependence, the level shifts that can be  detected are of smaller order
$\frac{d_n}{n}\sim \frac{cn ^{1-Dm/2} L(n)}{n} = c n^{-Dm/2} L(n)$; note that $Dm<1$.
}
\end{remark}

We will now apply Corollary~\ref{cor:Dn} in order to make power calculations for the change-point test that rejects for large values of $D_n$. Under the null hypothesis of no level shift,
\[
 D_n \claw \sup_{0\leq \lambda \leq 1} \frac{|a_m|}{m!} | \lambda Z_m(1)-Z_m(\lambda)|.
\]
If we denote by $q_\alpha$ the upper $\alpha$-quantile of the distribution of 
$ \sup_{0\leq \lambda\leq 1}  |\lambda Z_m(1)-Z_m(\lambda)|$,
we obtain
\[
  \lim_{n\rightarrow \infty} P_H \left(D_n\geq \frac{|a_m|}{m!} q_\alpha\right)
  =P\left(\sup_{0\leq \lambda \leq 1} \frac{|a_m|}{m!} |\lambda Z_m(1) -Z_m(\lambda)|
 \geq \frac{|a_m|}{m!} q_\alpha \right)=\alpha,
\]
where $P_H$ indicates the probability under the null hypothesis $H$.
Thus, the test that rejects the null hypothesis $H$ when $D_n\geq \frac{|a_m|}{m!} q_\alpha$ 
has asymptotic level $\alpha$.
If $h_n$ is chosen as in (\ref{eq:hn}), we obtain under the alternative $A_{\tau,h_n}(n)$
\begin{equation}
\lim_{n\rightarrow \infty} P_{A_{\tau,h_n}(n)}\left(D_n\geq \frac{|a_m|}{m!} q_\alpha\right) = P\left(\sup_{0\leq \lambda\leq 1}
  \left|\frac{a_m}{m!} (\lambda Z_m(1)-Z_m(\lambda))+c\, \phi_\tau (\lambda)\right| \geq \frac{|a_m|}{m!}  q_\alpha\right).
\label{eq:d-power}
\end{equation}
Thus, for large $n$, the power of our test at the alternative $A_{\tau,h_n}(n)$ is approximately given by 
the right-hand side of (\ref{eq:d-power}).

We may also apply Corollary~\ref{cor:Dn} in order to determine the size of a level shift at time $[\tau n]$ that can be detected with a given probability $\beta$.  First, we calculate $c=c(\alpha,\beta)$ such that
\begin{equation}
  P\left(\sup_{0\leq \lambda\leq 1}
  \left|\frac{a_m}{m!} (\lambda Z_m(1)-Z_m(\lambda))+c\, \phi_\tau (\lambda)\right| \geq \frac{|a_m|}{m!}
q_\alpha\right) =  \beta.
\label{eq:dif-h}
\end{equation}
Thus, by (\ref{eq:d-power}),  we get that  the asymptotic power of the test at the alternative $A_{\tau,h_n}(n)$ is equal to $\beta$. Thus, given a sample size $n$, we can detect a level shift of size $h_n=\frac{d_n}{n} c(\alpha,\beta)$ 
at time $[\tau n]$ with probability $\beta$ with a level $\alpha$ test based on the test statistic $D_n$.

\section{Power of the Wilcoxon Change-Point Test under Local Alternatives}
In the context of the Wilcoxon change-point test, the Hermite rank is not that of the function $G$, but of
the class of functions 
\begin{equation}
  1_{\{ G(\xi_i)\leq x \}}-F(x),\; x\in \R,
\label{eq:class}
\end{equation}
where $F(x)=E(1_{\{ G(\xi_i) \leq x \}}) =P(G(\xi_i)\leq x)$. We define the Hermite expansion of the class 
of functions (\ref{eq:class}) as
\[
  1_{\{ G(\xi_i) \leq x \}} -F(x) =\sum_{q=1}^\infty \frac{J_q(x)}{q!} H_q(\xi_i),
\]
where $H_q$ is again the $q$-th order Hermite polynomial and where the coefficients are
\begin{equation}
 J_q(x)=E\left( H_q(\xi_i)1_{\{G(\xi_i) \leq x \}}\right).
\label{eq:J-m}
\end{equation}
We define the Hermite rank of the class of functions (\ref{eq:class}) as
\[
  m:=\min\{q\geq 1: J_q(x)\neq 0 \mbox{ for some } x\in \R\}.
\]

\begin{theorem}
Suppose that $(\xi_i)_{i\geq 1}$ is a stationary Gaussian process with mean zero, variance $1$ and autocovariance function as in (\ref{eq:acf}) with $0\leq D<\frac{1}{m}$. Moreover, let $G:\R\rightarrow \R$ 
be a measurable function, and assume that $G(\xi_k)$ has continuous distribution function $F(x)$. Let
$m$ denote the Hermite rank of the class of functions (\ref{eq:class}), let $d_n$ be as in 
(\ref{eq:dn2}), and let the level shift $h_n$ be as in (\ref{eq:hn}). Then, under the  sequence of alternatives $A_{\tau, h_n}$,
defined in  (\ref{eq:local-alt}),
if $h_n\rightarrow 0$ as $n\rightarrow \infty$, the process
\begin{equation}
\left( \frac{1}{n\, d_n} \sum_{i=1}^{[n\lambda]} \sum_{j=[n\lambda]+1}^n 
\left( 1_{\{X_i\leq X_j\}} - \frac{1}{2} \right)
 -\frac{n}{d_n} \phi_\tau(\lambda) \int_\R (F(x+h_n)-F(x)) dF(x)\right)_{0\leq \lambda \leq 1}
\label{eq:wil-alt}
\end{equation}
converges in distribution towards the process 
\begin{equation}
 \left(\frac{\int_\R J_m(x) dF(x)}{m!}\big(Z_m(\lambda)-\lambda Z_m(1)\big) \right)_{0\leq \lambda \leq 1},
\label{eq:lim-null}
\end{equation}
where $(Z_m(\lambda))_{\lambda \geq 0}$ denotes the $m$-th order Hermite process with Hurst parameter
$H=1-\frac{Dm}{2}\in (\frac{1}{2},1)$ and where $J_m(x)$ is defined as in (\ref{eq:J-m}).
\label{th:wil-alt}
\end{theorem}

\begin{remark}{\rm
(i) The normalization $d_n$ and the processes $(Z_m(\lambda))_{\lambda \geq 0}$ in 
Theorem~\ref{th:dif-alt} and Theorem~\ref{th:wil-alt} are the same.
\\[1mm]
(ii) Since, by assumption, the distribution $F(x)$ of $G(\xi_k)$ is continuous, it follows from integration by parts that 
$\int_\R F(x) dF(x)=\frac{1}{2}$. This explains the $1/2$ in (\ref{eq:wil-alt}) because 
$\int_\R F(x) dF(x) =E(1_{\{X_1\leq X_1^\prime\}})$, where $X_1^\prime$  is an independent copy of $X_1$.
The independence assumption is reasonable as the dependence between $X_i$ and $X_j$ vanishes asymptotically when $|i-j|\rightarrow \infty$.
\\[1mm]
(iii)
As noted at the beginning of the proof, the first part  of (\ref{eq:wil-alt}) converges to (\ref{eq:lim-null}) under the null hypothesis. We show in the proof that the second part of (\ref{eq:wil-alt}) compensates for the 
presence of the alternatives $A_{\tau,h_n}$.
\\[1mm]
(iv) We make no assumption about the exact order of the sequence $(h_n)_{n\geq 1}$. Theorem~\ref{th:wil-alt} holds under the very general assumption that $h_n\rightarrow 0$, as $n\rightarrow \infty$. 
\\[1mm]
(v)
If we choose $(h_n)_{n\geq 1}$ as in (\ref{eq:hn}),  the centering constants in (\ref{eq:wil-alt}) converge, provided some technical assumptions are satisfied. To see this, observe that
\begin{eqnarray*}
\frac{n}{d_n}\phi_\tau(\lambda) \int_\R (F(x+h_n)-F(x)) dF(x)
 &\sim &\frac{n\, h_n}{d_n} \phi_\tau(\lambda) \int_\R\frac{F(x+h_n)-F(x)}{h_n} dF(x) \\
 &\rightarrow & c\, \phi_\tau(\lambda) \int_\R f(x) dF(x)\\
 &=& c\, \phi_\tau(\lambda) \int_\R f^2(x) dx.
\end{eqnarray*}
The convergence in the next to last step requires some justification -- this holds, e.g. if $F$ is differentiable with bounded derivative $f(x)$. 
}
\end{remark}
\begin{corollary}
Suppose that $(\xi_i)_{i\geq 1}$ is a stationary Gaussian process with mean zero, variance $1$ and autocovariance function as in (\ref{eq:acf}) with $0\leq D<\frac{1}{m}$. Moreover, let $G:\R\rightarrow \R$ 
be a measurable function, and assume that $G(\xi_k)$ has a distribution function $F(x)$ with bounded density $f(x)$. Let $m$ denote the Hermite rank of the class of functions 
$1_{\{G(\xi_i)\leq x  \}} -F(x)$, $x\in \R$. Then, under the  sequence of alternatives $A_{\tau, h_n}$,
defined in  (\ref{eq:local-alt}),
with $h_n \sim c\frac{d_n}{n}$  we obtain
that
\begin{equation}
\left(\frac{1}{n\, d_n} \sum_{i=1}^{[n\lambda]} \sum_{j=[n\lambda]+1}^n 
\left( 1_{\{X_i\leq X_j\}} - \frac{1}{2} \right)\right)_{0\leq \lambda \leq 1}
\label{eq:wil-alt2}
\end{equation}
converges in distribution to the process
\[
 \left( \frac{\int_\R J_m(x) dF(x)}{m!}(Z_m(\lambda)-\lambda Z_m(1))  + c\phi_\tau(\lambda)
 \int_\R f^2(x) dx\right)_{0\leq \lambda \leq 1}.
\]
\label{cor:wil-alt}
\end{corollary}

{\em Proof of Theorem~\ref{th:wil-alt}.}  In our proof, we will make use of the limit theorem that was derived in Dehling, Rooch and Taqqu (2013) under the null hypothesis. They showed (see Theorem~1) that
\begin{equation}
\frac{1}{n\, d_n} \sum_{i=1}^{[n\lambda]} \sum_{j=[n\lambda]+1}^n 
\left( 1_{\{G(\xi_i)\leq G(\xi_j)\}} - \frac{1}{2} \right)
\rightarrow \frac{\int_\R J_m(x) dF(x)}{m!} (Z_m(\lambda)-\lambda Z_m(1)) .
\label{eq:wil-hyp}
\end{equation}
In order to make use of this result, we will decompose the test statistic into a term whose distribution is the same both under the null hypothesis as well as under the alternative, and a second term which, after proper centering converges to zero. As in Dehling, Rooch and Taqqu (2013), we will express the test statistic as a functional of the empirical distribution function of the $G(\xi_i)$, namely
\[
 F_k(x)=\frac{1}{k}\sum_{i=1}^k 1_{\{G(\xi_i)\leq x  \}}.
\]
Given integers $k,l$ with $k\leq l$ we denote by $F_{k,l}(x)$ the empirical distribution function based on  $G(\xi_k),\ldots, G(\xi_l)$, i.e.
\begin{equation}
 F_{k,l}(x)=\frac{1}{l-k+1} \sum_{i=k}^l 1_{\{G(\xi_i)\leq x\}}.
\label{eq:edf-fkl}
\end{equation}
Recall that under the local alternative, we have
\[
 X_i= \left\{ \begin{array}{ll} 
 G(\xi_i)+\mu  & \mbox{ for } i=1,\ldots,[n\tau]\\
 G(\xi_i)+\mu+h_n & \mbox{ for } i=[n\tau]+1,\ldots,n.
\end{array}
\right.
\]
Thus, we obtain for $\lambda \leq \tau$,
\begin{eqnarray}
\sum_{i=1}^{[n\lambda]} \sum_{j=[n\lambda]+1}^n 1_{\{X_i\leq X_j\}}
&=& 
\sum_{i=1}^{[n\lambda]} \sum_{j=[n\lambda]+1}^{[n\tau]} 1_{\{G(\xi_i)+\mu \leq G(\xi_j)+\mu  \}} 
+ \sum_{i=1}^{[n\lambda]} \sum_{j=[n\tau]+1}^n 1_{\{G(\xi_i)+\mu \leq G(\xi_j)+\mu+h_n  \}} \nonumber\\
&=& \sum_{i=1}^{[n\lambda]} \sum_{j=[n\lambda]+1}^{[n\tau]} 1_{\{G(\xi_i) \leq G(\xi_j)  \}}
 + \sum_{i=1}^{[n\lambda]} \sum_{j=[n\tau]+1}^n 1_{\{G(\xi_i) \leq G(\xi_j) +h_n\}}\nonumber\\
&=& \sum_{i=1}^{[n\lambda]} \sum_{j=[n\lambda]+1}^n 1_{\{G(\xi_i) \leq G(\xi_j)\}} 
+ \sum_{i=1}^{[n\lambda]} \sum_{j=[n\tau]+1}^n 
 \left( 1_{\{G(\xi_i) \leq G(\xi_j) +h_n\}} - 1_{\{G(\xi_i) \leq G(\xi_j)\}}
\right) \nonumber \\
&=& \sum_{i=1}^{[n\lambda]} \sum_{j=[n\lambda]+1}^n 1_{\{G(\xi_i) \leq G(\xi_j)\}} 
+ \sum_{i=1}^{[n\lambda]} \sum_{j=[n\tau]+1}^n 
   1_{\{G(\xi_j) < G(\xi_i) \leq G(\xi_j) +h_n\}}.
\end{eqnarray}
In the same way, we obtain for $\lambda \geq \tau$,
\begin{eqnarray}
&& \sum_{i=1}^{[n\lambda]} \sum_{j=[n\lambda]+1}^n 1_{\{X_i \leq  X_j  \}} \nonumber\\
&&\quad =\sum_{i=1}^{[n\tau]} \sum_{j=[n\lambda]+1}^n 
 1_{\{G(\xi_i)+\mu\leq G(\xi_j)+\mu+h_n  \}}
+\sum_{i=[n\tau]+1}^{[n\lambda]} \sum_{j=[n\lambda]+1}^n 1_{\{G(\xi_i)+\mu+h_n \leq G(\xi_j)+\mu+h_n  \}} \nonumber \\
&& \quad = \sum_{i=1}^{[n\tau]} \sum_{j=[n\lambda]+1}^n 1_{\{G(\xi_i) \leq G(\xi_j)+h_n  \}}
+\sum_{i=[n\tau]+1}^{[n\lambda]} \sum_{j=[n\lambda]+1}^n 1_{\{G(\xi_i) \leq G(\xi_j)  \}} \nonumber \\
&& \quad  = \sum_{i=1}^{[n\lambda]} \sum_{j=[n\lambda]+1}^n 1_{\{G(\xi_i) \leq G(\xi_j)\}} 
 + \sum_{i=1}^{[n\tau]} \sum_{j=[n\lambda]+1}^n 
 \left( 1_{\{G(\xi_i) \leq G(\xi_j) +h_n\}} - 1_{\{G(\xi_i) \leq G(\xi_j)\}} \right) \nonumber\\
&& \quad = \sum_{i=1}^{[n\lambda]} \sum_{j=[n\lambda]+1}^n 1_{\{G(\xi_i) \leq G(\xi_j)\}} 
 + \sum_{i=1}^{[n\tau]} \sum_{j=[n\lambda]+1}^n 1_{\{G(\xi_j)< G(\xi_i) \leq G(\xi_j) +h_n\}} 
\end{eqnarray}
Thus, in order to prove Theorem~\ref{th:wil-alt}, it suffices to show that the following two terms,
\begin{equation}
 \frac{1}{n\, d_n} \sup_{0\leq \lambda \leq \tau}
 \left|  
 \sum_{i=1}^{[n\lambda]} \sum_{j=[n\tau]+1}^n \!
 1_{\{G(\xi_j)< G(\xi_i) \leq G(\xi_j) +h_n\}} 
 -n^2 \lambda (1\!-\!\tau) \int_\R (F(x+h_n)\!-\!F(x)) dF(x)
\right| 
\label{eq:rem-wil1}
\end{equation}
and 
\begin{equation}
 \frac{1}{n\, d_n} \sup_{\tau \leq \lambda \leq 1}
 \left|  
 \sum_{i=1}^{[n\tau]} \sum_{j=[n\lambda]+1}^n \!
 1_{\{G(\xi_j)< G(\xi_i) \leq G(\xi_j) +h_n\}} 
 -n^2 \tau(1\!-\!\lambda) \int_\R (F(x+h_n)\!-\!F(x)) dF(x)
\right| 
\label{eq:rem-wil2}
\end{equation}
both converge to zero in probability. We first show this for (\ref{eq:rem-wil1}). Observe that 
\begin{eqnarray*}
&& \sum_{i=1}^{[n\lambda]} \sum_{j=[n\tau]+1}^n \! \! 
 1_{\{G(\xi_j)< G(\xi_i) \leq G(\xi_j) +h_n\}} -n^2\lambda(1-\tau)
\int_\R \left( F(x+h_n) - F(x) \right)dF(x) \\
&&= [n\lambda] \sum_{j=[n\tau]+1}^n \left( F_{[n\lambda]} (G(\xi_i)+h_n) -F_{[n\lambda]}(G(\xi_i)) \right)  \\
&& \quad -n^2\lambda(1-\tau) \int_\R \left( F(x+h_n) - F(x) \right)dF(x)  \\
&&=[n\lambda](n\!-\![n\tau]) \int_\R \left( F_{[n\lambda]}(x+h_n)-F_{[n\lambda]}(x) \right)
 dF_{[n\tau]+1,n}(x)\\
&& \quad -n^2\lambda(1-\tau) \int_\R \left( F(x+h_n) - F(x) \right)dF(x) \\
&&=[n\lambda](n\!-\![n\tau])\left(\int_\R  (F_{[n\lambda]}(x+h_n)-F_{[n\lambda]}(x)) 
 dF_{[n\tau]+1,n}(x) -\int_\R (F(x+h_n) - F(x)) dF(x) \right) \\
&&\quad + \left( [n\lambda](n-[n\tau]) -n^2 \lambda (1-\tau) \right)\int_\R (F(x+h_n) - F(x)) dF(x) 
\end{eqnarray*}
Note that $|[n\lambda](n-[n\tau]) -n^2 \lambda (1-\tau)|\leq n$ and $|\int_\R (F(x+h_n)-F(x)) dF(x)|\leq 1$. Thus 
\[
 \frac{1}{n\, d_n} \left( [n\lambda](n-[n\tau]) -n^2 \lambda (1-\tau) \right)\int_\R (F(x+h_n) - F(x)) dF(x) 
 \leq \frac{1}{d_n} \rightarrow 0,
\]
as $n\rightarrow \infty$. Hence, in order to show that (\ref{eq:rem-wil1}) converges to zero in probability, it suffices to show that 
\begin{equation}
 \frac{1}{n\, d_n}[n\lambda](n-[n\tau])\left(\int_\R  (F_{[n\lambda]}(x+h_n)-F_{[n\lambda]}(x)) 
 dF_{[n\tau]+1,n}(x) -\int_\R (F(x+h_n) - F(x)) dF(x) \right) 
\label{eq:rem-wil1a}
\end{equation} 
converges to zero, in probability.
In order to prove this, we rewrite the difference of the integrals in (\ref{eq:rem-wil1a}) as
\begin{eqnarray}
&& \int_\R \left( F_{[n\lambda]}(x+h_n) - F_{[n\lambda]}(x) \right) dF_{[n\tau]+1,n}(x)
 -\int_\R\left(F(x+h_n)-F(x)  \right) dF(x) 
\label{eq:final} \\
&& \quad=\int_\R\left( F_{[n\lambda]}(x+h_n)-F_{[n\lambda]}(x)  \right)
 -\left( F(x+h_n)-F(x) \right) dF_{[n\tau]+1,n} (x) \nonumber\\
 &&\qquad + \int_\R \left( F(x+h_n)-F(x) \right) d(F_{[n\tau]+1,n}-F)(x)\nonumber\\
 &&\quad=\int_\R\left( F_{[n\lambda]}(x+h_n)-F_{[n\lambda]}(x)  \right)
 -\left( F(x+h_n)-F(x) \right) dF_{[n\tau]+1,n} (x) \nonumber\\
  &&\qquad - \int_\R \left(F_{[n\tau]+1,n}(x)-F(x)   \right) d( F(x+h_n)-F(x)),
\nonumber
\end{eqnarray}
where we have used integration by parts in the final step. 
Thus, in order to prove that (\ref{eq:rem-wil1a}) converges to zero, it suffices to show that 
the following two terms converge in probability, as $n\rightarrow 0$,
\begin{eqnarray}
&& \frac{1}{d_n} [n\lambda] \int_\R \left( (F_{[n\lambda]}(x+h_n)-F_{[n\lambda]}(x) )
 -(F(x+h_n)-F(x))  \right) dF_{[n\tau]+1,n}(x) \rightarrow 0
\label{eq:rem-wil1aa} \\
&& \frac{1}{d_n} (n-[n\tau])
\int_\R \left( F_{[n\tau]+1,n}(x) -F(x) \right) d(F(x+h_n)-F(x)) \rightarrow 0.
\label{eq:rem-wil1ab}
\end{eqnarray}
In order to prove (\ref{eq:rem-wil1aa}) and (\ref{eq:rem-wil1ab}),  we now apply
the empirical process non-central limit theorem of Dehling and Taqqu (1989)
which states that
\[
 \left(d_n^{-1} [n\lambda] (F_{[n\lambda]}(x) - F(x))
 \right)_{x\in [-\infty, \infty], \lambda \in [0,1]}
 \claw (J(x) Z(\lambda))_{x\in [-\infty,\infty], \lambda\in [0,1]},
\]
where 
\[
 J(x)=J_m(x)=E\left( 1_{\{G(\xi_i)\leq x\}} H_m(\xi_i) \right)  \quad \mbox{and} \quad Z(\lambda)=\frac{Z_m(\lambda)}{m!}.
\]
By the Dudley-Wichura version of Skorohod's representation theorem
(see Shorack and Wellner (1986), Theorem~2.3.4) we
may assume without loss of generality that convergence holds 
almost surely 
with respect to the supremum norm on the function space 
$D([0,1]\times [-\infty,\infty])$, i.e. 
\begin{equation}
\sup_{\lambda\in [0,1], x\in \R }  \left| d_n^{-1} [n\lambda] (F_{[n\lambda]}(x) - F(x)) 
 -J(x) Z(\lambda)\right| \longrightarrow  0 \quad a.s.
 \label{eq:s-repr1} 
\end{equation}
Note that by definition, for any $\lambda \leq \tau$
\[
 ([n\tau]-[n\lambda]) (F_{[n\lambda]+1, [n\tau]} (x) -F(x))
 =[n\tau](F_{[n\tau]}(x) -F(x))-[n\lambda](F_{[n\lambda]}(x) - F(x)).
\]
Hence, we may deduce from (\ref{eq:s-repr1}) the following limit theorem for the empirical distribution of
the observations $X_{[n\lambda]+1},\ldots, X_{[n\tau]}$,
\begin{equation}
\sup_{0\leq \lambda \leq \tau, x\in \R}  
\left|d_n^{-1} ([n\tau]-[n\lambda]) (F_{[n\lambda]+1,[n\tau]}(x)-F(x)) -J(x)(Z(\tau)-Z(\lambda))
\right| \rightarrow 0,
\label{eq:s-repr2}
\end{equation}
almost surely.   
In the same way, we obtain 
\begin{equation}
\sup_{0\leq \lambda \leq 1, x\in \R} \left| d_n^{-1}(n-[n\lambda]) (F_{[n\lambda]+1,n}-F(x)) -
J(x)(Z(1)-Z(\lambda))  \right| \rightarrow 0, 
\label{eq:s-repr3}
\end{equation}
almost surely. Now we return to (\ref{eq:rem-wil1aa}) and write
\begin{eqnarray}
&& \left| \int_\R \frac{1}{d_n}[n\lambda]\left( (F_{[n\lambda]}(x+h_n)-F_{[n\lambda]} (x))
 -(F(x+h_n)-F(x))\right) dF_{[n\tau]+1,n}(x)  \right| \label{eq:rem-wil1aaf} \\
&&\quad \leq \left| \int_\R (J(x+h_n)-J(x)) Z(\lambda) dF_{[n\tau]+1,n}(x)  \right| \nonumber \\
 && \qquad +
 \sup_{x\in \R, 0\leq \lambda\leq 1} 
 \bigg|\frac{1}{d_n}[n\lambda] \big( (F_{[n\lambda]}(x+h_n) 
 - F_{[n\lambda]}( x)) -(F(x+h_n)-F(x))\big)\nonumber \\ 
&& \qquad
 -(J(x+h_n)-J(x))Z(\lambda) \bigg| \nonumber \\
 &&\quad \leq  \left|\int_\R (J(x+h_n)-J(x)) dF_{[n\tau]+1,n}(x)  \right|
 \, \sup_{0\leq \lambda \leq 1} |Z(\lambda)|
\nonumber\\
&& \qquad +
 \sup_{x\in \R, 0\leq \lambda\leq 1} \bigg|\frac{1}{d_n}[n\lambda] \left(
(F_{[n\lambda]}(x+h_n) - F_{[n\lambda]}( x)) - (F(x+h_n)-F(x)) \right)  
\nonumber \\
 &&\qquad  -(J(x+h_n)-J(x))Z(\lambda) \bigg|. \nonumber 
\end{eqnarray}
The second term on the right-hand side converges to zero by (\ref{eq:s-repr1}). Concerning the first term, note that 
\begin{equation}
J(x)=\int_\R 1_{\{G(y)\leq x  \}} H_m(y) \phi(y) dy 
 = - \int_\R 1_{\{x\leq G(y)  \}} H_m(y) \phi(y) dy,
\label{eq:j}
\end{equation}
where $\phi(y)=\frac{1}{\sqrt{2\pi}}  e^{-y^2/2}$ denotes the standard normal density function. For the second identity, we have used the fact that $G(\xi)$, by assumption, has a continuous distribution, and that $\int_\R H_m(y)\phi(y) dy=0$, for $m\geq 1$. Using (\ref{eq:j})  we thus obtain
\begin{eqnarray}
 \int_\R J(x) dF_{[n\tau]+1,n} (x)
 &=& -\int_\R \int_\R 1_{\{x\leq G(y)  \}} H_m(y) \phi(y) dy dF_{[n\tau]+1,n}(x) 
\label{eq:j-int1}\\
 &=& - \int_\R\int_\R 1_{\{x\leq G(y)\}} dF_{[n\tau]+1,n}(x) H_m(y) \phi(y) dy\nonumber \\
 &=& - \int_\R F_{[n\tau]+1,n} (G(y)) H_m(y) \phi(y) dy, \nonumber
\end{eqnarray}
and, using analogous arguments,
\begin{equation}
 \int_\R J(x+h_n) dF_{[n\tau]+1,n}(x) = - \int_\R 
 F_{[n\tau]+1,n} (G(y)-h_n) H_m(y) \phi(y) dy.
\label{eq:j-int2}
\end{equation}
By the Glivenko-Cantelli theorem, applied to the stationary, ergodic process $(G(\xi_i))_{i\geq 1}$, we get
$ \sup_{x\in \R} \left|F_n(x)-F(x)  \right| \rightarrow 0$, almost surely. Since
\[
 F_{[n\tau]+1,n}(x) =\frac{n}{n-[n\tau]} F_n(x) - \frac{[n\tau]}{n-[n\tau]} F_{[n\tau]}(x),
\]
we get that, almost surely,
\begin{equation}
 \sup_{x\in \R} \left|F_{[n\tau]+1,n} (x) -F(x)  \right| \rightarrow 0.
\label{eq:gc}
\end{equation}
Returning to the first term on the right-hand side of (\ref{eq:rem-wil1aaf}), we obtain, using
(\ref{eq:j-int1}) and (\ref{eq:j-int2}),
\begin{eqnarray*}
&& \left| \int_\R \left( J(x+h_n)-J(x) \right) dF_{[n\tau]+1,n} (x)  \right|\\
&&\quad= 
 \left| \int_\R \left(F_{[n\tau]+1,n}(G(y)-h_n)-F_{[n\tau]+1,n}(G(y))\right) 
H_m(y) \phi(y) dy \right| \\
&& \quad \leq  \int_\R \left| F(G(y) -h_n)-F(G(y)) \right| \left|H_m(y)  \right| \phi(y) dy \\
 &&\qquad + 2\sup_{x} \left|F_{[n\tau]+1,n}(x)-F(x) \right|
 \int_\R \left| H_m(y)   \right|\phi(y) dy.
 \end{eqnarray*}
Both terms on the right-hand-side converge to zero; the second one by 
(\ref{eq:gc}), the first one by continuity of $F$, the fact that $h_n\rightarrow 0$, and Lebesgue's dominated 
convergence theorem. In both cases, we have made use of the fact that 
$\int |H_m(y)|\phi(y)dy<\infty$. Thus we have finally established 
(\ref{eq:rem-wil1aa}). 
In order to prove (\ref{eq:rem-wil1ab}), we observe that
\begin{eqnarray*}
&& \frac{1}{d_n}(n-[n\tau]) \int_\R \left( F_{[n\tau]+1,n}(x)-F(x) \right) 
d(F(x+h_n)-F(x)) \\
&&\quad \leq 
\left| \int_\R J(x)(Z(\tau)-Z(1))d(F(x+h_n)-F(x))\right| \\
&& \qquad +\sup_{x\in \R} \left| \frac{1}{d_n}(n-[n\tau])
(F_{[n\tau]+1,n}(x) -F(x)) - J(x)(Z(\tau)-Z(1))  \right|\\
&&\quad \leq \left| \int_\R J(x) d(F(x+h_n)-F(x))  \right| \, |Z(\tau)-Z(1)|
\\
&& \qquad +\sup_{x\in \R} \left| \frac{1}{d_n}(n-[n\tau])
(F_{[n\tau]+1,n}(x) -F(x)) - J(x)(Z(\tau)-Z(1))  \right|.
\end{eqnarray*}
The second term on the right-hand side converges to zero, by (\ref{eq:s-repr3}). Concerning the first term, note that
\[
 \int_\R J(x) d(F(x+h_n)-F(x)) = E\left( J(G(\xi_i) - h_n)-J(G(\xi_i))\right).
\]
Applying Lebesgue's dominated convergence theorem and making use of the fact that, by assumption, $J$ is continuous, we obtain that $\int_\R J(x) d(F(x+h_n)-F(x)) \rightarrow 0$. In this way, we  have finally proved that (\ref{eq:rem-wil1}) converges to zero, in probability. By similar arguments, we can prove this for (\ref{eq:rem-wil2}), which finally ends the proof of Theorem~\ref{th:wil-alt}.
\hfill $\Box$

\section{ARE of the Wilcoxon Change-Point Test and the CUSUM Test for LRD Data}

In this section, we calculate the asymptotic relative efficiency (ARE) of the Wilcoxon change-point test 
with respect to the CUSUM test. To do so, we calculate the number of observations needed to detect a small level shift $h$ at time $[\tau\,n]$ with  a test of given level $\alpha$ and given power $\beta$, both for the Wilcoxon change-point test and the CUSUM test, and denote these numbers by $n_W$ and $n_C$, respectively. We then define the asymptotic relative efficiency of the Wilcoxon change-point test with respect to the CUSUM test by
\begin{equation}
 ARE(W,C)=\lim_{h\rightarrow 0}\frac{n_C}{n_W}.
\label{eq:are}
\end{equation}
It will turn out that the limit (\ref{eq:are}) exists and that the asymptotic relative efficiency does not depend on the choice of $\tau,\alpha,\beta$. If this limit is larger than $1$, then the CUSUM test requires a larger
sample size to detect the level shift, and hence the Wilcoxon change-point test is (asymptotically) more efficient.

In the remaining part of this section, we will focus on the case when $m=1$ both for the 
CUSUM as well as the Wilcoxon change-point test, i.e. when the Hermite rank of $G(\xi_1)$ and of
the class of functions $1_{\{G(\xi_1)\leq x\}}-F(x), x\in \R$, are both equal to $1$. This is the case, for example,  when $G$ is a strictly monotone function. In this case
\[
  \int_\R J_1(x) dF(x)=-\frac{1}{2\sqrt{\pi}},
\]
see Relation (20) in Dehling, Rooch and Taqqu (2013), showing that the Hermite rank of the class of functions  
$1_{\{G(\xi_1)\leq x\}}-F(x), x\in \R$ equals $1$. Focusing now on $G(\xi_1)$ and using integration by parts, we get that the first order Hermite coefficient $a_1$ of $G$ equals
\[
a_1=E(G(\xi_1)\xi_1) 
=\int_\R G(x) x \phi(x) dx
=-\int_\R G(x) \phi^\prime(x) dx 
= \int_\R \phi(x) dG(x) >0,
\]
where $\phi(x)=\frac{1}{\sqrt{2 \pi}}e^{-x^2/2}$ denotes the standard normal density function.
Thus, the Hermite rank of $G(\xi_i)$ equals $1$, as well.

In this case, i.e. when $m=1$, the Hermite process arising as limit in Theorem~\ref{th:dif-alt}, Theorem~\ref{th:wil-alt} and Corollary~\ref{cor:wil-alt}
is fractional Brownian motion $(B_H(\lambda))_{0\leq \lambda \leq 1}$. Note that fractional Brownian motion is symmetric, i.e. $(-B_H(\lambda))_{0\leq \lambda \leq 1}$ has the same distribution as $(B_H(\lambda))_{0\leq \lambda \leq 1}$. Thus the limit processes in Theorem~\ref{th:dif-alt} and Corollary~\ref{th:wil-alt} can also be written as
\begin{equation}
\left(|a_1|(B_H(\lambda)-\lambda B_H(1) ) +c\, \phi_\tau(\lambda)\right)_{0\leq \lambda \leq 1},
\label{eq:dif-alt2}
\end{equation}

\begin{equation}
\left(\left|\int_\R J_1(x) dF(x)\right|\left(B_H(\lambda)-\lambda B_H(1) \right) +c\, \phi_\tau(\lambda)
\int_\R f^2(x) dx \right)_{0\leq \lambda \leq 1}.
\label{eq:wil-alt2}
\end{equation}

As preparation, we first calculate a quantity that is related to the asymptotic relative efficiency, namely the ratio of the sizes of level shifts that can be detected by the two tests, based on the same number of observations $n$, again for given values of $\tau,\alpha,\beta$. We denote the corresponding level shifts by $h_W(n)$ and $h_C(n)$, respectively, assuming that these numbers depend on $n$ in the following way:
\begin{eqnarray}
 h_W(n)&=&c_W\frac{d_n}{n}
\label{eq:h-w} \\
 h_C(n)&=&c_C\frac{d_n}{n}.
\label{eq:h-c}
\end{eqnarray}
In order to simplify the following considerations, we take a one-sided change-point test, thus rejecting the hypothesis of no change-point for large values of 
\[
 \max_{1\leq k \leq n-1} \sum_{i=1}^k\sum_{j=k+1}^n (X_j-X_i) 
\]
or
\[
\max_{1\leq k \leq n-1} \sum_{i=1}^k \sum_{j=k+1}^n 1_{\{X_i\leq X_j \}},
\]
respectively. These are the appropriate tests when testing against the alternative of a non-negative level shift.
In order to obtain tests that have asymptotically level $\alpha$, the CUSUM test and the Wilcoxon
change-point test reject the null-hypothesis when 
\begin{eqnarray}
  && \frac{1}{n\, d_n\, |a_1|} \max_{1\leq k\leq n-1} \sum_{i=1}^k\sum_{j=k+1}^n (X_j-X_i) \geq q_\alpha,
\label{eq:test-dif}\\
 && \frac{1}{n\, d_n \, |\int_\R J_1(x) dF(x)|} \max_{k=1,\ldots,n-1} \sum_{i=1}^k \sum_{j=k+1}^n
 \left( 1_{\{X_i\leq X_j  \}} -\frac{1}{2} \right) \geq q_\alpha,
\label{eq:test-wil}
\end{eqnarray}
where $q_\alpha$ denotes the upper $\alpha$ quantile of the distribution of 
$\sup_{0\leq \lambda \leq 1} (B_H(\lambda)-\lambda B_H(1))$. This follows from Theorem~\ref{th:dif-alt}
and Corollary~\ref{cor:wil-alt} after applying the continuous mapping theorem. 
The constants $a_1$ and the functions $J_1(x)$ are defined in (\ref{eq:a-m}) and (\ref{eq:J-m}),
respectively, and have just been computed.
Under the sequence of alternatives
$A_{\tau,h_C(n)}$,  the asymptotic distribution of the test statistic in (\ref{eq:test-dif})  is given by
\[
 \sup_{0\leq \lambda \leq 1} \left(B_H(\lambda)-\lambda B_H(1) +\frac{c_C}{|a_1|}\phi_\tau(\lambda)
 \right);
\]
see Theorem~\ref{th:dif-alt}.  Under the sequence of alternatives $A_{\tau,h_W(n)}$, the asymptotic distribution of the test statistic in (\ref{eq:test-wil}) is given by
\[
 \sup_{0\leq \lambda \leq 1} \left(B_H(\lambda)-\lambda B_H(1) 
 +\frac{c_W\, \int_\R f^2(x) dx}{|\int_\R J_1(x) dF(x)|}\phi_\tau(\lambda)
 \right);
\]
see Corollary~\ref{cor:wil-alt}.
Thus, the asymptotic power of the CUSUM test is given by 
\begin{eqnarray}
&& \lim_{n\rightarrow\infty}
 P_{A_{\tau,h_C(n)}} \left(    \frac{1}{n\, d_n\,| a_1|} \max_{1\leq k\leq n-1} \sum_{i=1}^k\sum_{j=k+1}^n (X_j-X_i) \geq q_\alpha   \right)\nonumber \\
 && 
 \qquad =P\left( \sup_{0\leq \lambda \leq 1} \left(B_H(\lambda)-\lambda B_H(1) +\frac{c_C}{|a_1|}\phi_\tau(\lambda)
 \right)\geq q_\alpha \right).
\label{eq:power-D}
\end{eqnarray}
In the same way, we obtain the power of the Wilcoxon change-point test
\begin{eqnarray}
&& \lim_{n\rightarrow\infty}
 P_{A_{\tau,h_W(n)}} \left(    \frac{1}{n\, d_n \, |\int_\R J_1(x) dF(x)|} \max_{1\leq k\leq n-1} \sum_{i=1}^k \sum_{j=k+1}^n
 \left( 1_{\{X_i\leq X_j  \}} -\frac{1}{2} \right) \geq q_\alpha
   \right)\nonumber \\
&&\qquad =P\left( 
 \sup_{0\leq \lambda \leq 1} \left(B_H(\lambda)-\lambda B_H(1) 
 +\frac{c_W\,  \int_\R f^2(x) dx}{|\int_\R J_1(x) dF(x)|}\phi_\tau(\lambda)
 \right)
\geq q_\alpha \right).
\label{eq:power-W}
\end{eqnarray}
Thus, if we want the two tests to have identical power, we have to choose $c_C$ and $c_W$ in such a way 
that
\[
  \frac{c_C}{|a_1|}\phi_\tau(\lambda)=\frac{c_W\, \int_\R f^2(x) dx}{|\int_\R J_1(x) dF(x)|}\phi_\tau(\lambda),
\]
which again yields by (\ref{eq:h-w}) and (\ref{eq:h-c}), 
\begin{equation}
 \frac{h_C(n)}{h_W(n)}=\frac{c_C}{c_W}=\frac{|a_1| \int_\R f^2(x) dx}{|\int_\R J_1(x) dF(x)|}.
\label{eq:h-ratio}
\end{equation}
This quantity gives the ratio of the height of a level shift that can be detected by a CUSUM test over the height that can be detected by a Wilcoxon change-point test, when both tests are assumed to have the same level $\alpha$, the same power $\beta$ and the shifts are taking place at the same time $[n\tau]$. In addition, we assume that the tests are based on the same number of observations $n$, which is supposed to be large.

\begin{example}{\rm
In the case of Gaussian data, i.e.\ when $G(\xi)=\xi$, we have $m=1$, $a_1=E(\xi_1H_1(\xi_1))=E(\xi_1^2)=1$, $\int_\R f^2(x) dx =\int_\R \frac{1}{2\pi} e^{-x^2} dx 
=\frac{1}{2\sqrt{\pi}}$ and 
$\int_\R J_1(x) dF(x) = -\frac{1}{2\sqrt{\pi}}$; see Dehling, Rooch and Taqqu (2013), Relation (20). Thus we obtain
\begin{equation}
\frac{c_C}{c_W}=\frac{1/2\sqrt{\pi}}{1/2\sqrt{\pi}} =1.
\label{eq:are-gauss}
\end{equation} 
Hence, both tests can asymptotically, as $n\rightarrow \infty$, detect level shifts of the same height.
}
\label{ex:are-gauss}
\end{example}

The level shifts can be expressed in terms of 
\[
  \psi(t):= P\left(\sup_{0\leq \lambda\leq 1} (B_H(\lambda)-\lambda B_H(1) +t \, \phi_\tau(\lambda))
 \geq q_\alpha \right),
\] 
viewed as a function of $t$, for fixed values of $\tau$ and $\alpha$. The function $\psi$ is monotonically increasing. We define the generalized inverse,
\[
 \psi^{-}(\beta):=\inf \{t \geq 0: \psi(t)\geq \beta\}.
\]
Thus, we get 
\begin{equation}
 P\left(\sup_{0\leq \lambda\leq 1} \left(B_H(\lambda)-\lambda B_H(1) +\psi^{-}(\beta) \, \phi_\tau(\lambda)\right) 
 \geq q_\alpha\right) \geq \beta,
\label{eq:beta}
\end{equation}
and, in fact, for given $\tau$, $\alpha$ and $\beta$, $\psi^{-}(\beta)$ is the smallest number having this property. 

We can now apply Theorem~\ref{th:dif-alt} and Theorem~\ref{th:wil-alt}.
By comparing (\ref{eq:power-D}) and (\ref{eq:beta}), 
we can detect a level shift of size $h$ at time $[n\tau]$ with a CUSUM test of level $\alpha$ and power $\beta$ based on $n$ observations, if 
$h_C(n)\sim\frac{d_n}{n}\, c_C$,  where $c_C$ satisfies $\frac{c_C}{|a_1|}=\psi^{-}(\beta)$. Hence we obtain that 
$h_C(n)$ has to satisfy
\[
 h_C(n)=\frac{d_n}{n}|a_1| \psi^{-}(\beta) = n^{-D/2} L(n) |a_1|  \psi^{-}(\beta).
\]
Similarly, by comparing (\ref{eq:power-W}) and (\ref{eq:beta}), we get for the Wilcoxon
change-point test that $n$ has to satisfy
\[
  h_W(n)= n^{-D/2}L(n) \frac{|\int_\R J_1(x) dF(x)|}{\int_\R f^2(x) dx} \psi^{-}(\beta).
\]

In the following theorem, we compute the asymptotic relative efficiency of the Wilcoxon change point test 
with respect to the CUSUM test.

\begin{theorem} Let $(\xi_i)_{i\geq 1}$ be a stationary Gaussian process with mean zero, variance $1$ and
autocovariance function as in (\ref{eq:acf}). Moreover, let $G:\R\rightarrow \R$ be a measurable
function satisfying $E(G^2(\xi_1))<\infty$, and such that $G(\xi_1)$ has a distribution function $F(x)$ with
bounded density $f(x)$. Assume that the Hermite rank of $G(\xi_1)$ as well as the Hermite rank of the
class of functions $(1_{\{G(\xi_1)\leq x \}} -F(x))$, $x\in \R$ are equal to $1$.
Moreover assume that $0<D<1$. Then
\begin{equation}
  ARE(T_W, T_C)=\left(\frac{|a_1|\, \int_\R f^2(x) dx }{|\int_\R J_1(x) dF(x)|} \right)^{2/ D},
\label{eq:are-formula}
\end{equation}
where $T_C$ and $T_W$ denote the CUSUM test and the Wilcoxon change-point test, respectively.
\end{theorem}

{\em Proof.}
For abbreviation, we define
\begin{equation}
  b=\left( \frac{|\int_\R J_1(x) dF(x)|}{|a_1| \int_\R f^2(x) dx} \right)^{2/D}.
\label{eq:bD}
\end{equation}
We will show that the Wilcoxon change-point test based on $b\,n$ observations has asymptotically the 
same power as the CUSUM test based on $n$ observations. We will consider the local alternative 
\begin{equation}
 A_n^C=A_{\tau,h_n^C} = A_{\tau,c\, \frac{d_n}{n}}(n)
\label{eq:C-local}
\end{equation}
for the CUSUM test, and the local alternative
\begin{equation}
  A_n^W=A_{\tau,h_n^W}=A_{\tau, c\, \frac{b}{n} d_{n/b}}(n)
\label{eq:w-local}
\end{equation}
for the Wilcoxon change-point test. Note that under $A_{bn}^W$ the level shift is the same as under $A_n^C$. Further observe that, by (\ref{eq:dn2}),
\begin{eqnarray}
h_n^W\sim c\, \frac{b}{n} d_{n/b} 
&=& c\, \frac{b}{n} \kappa_1^{1/2} (n/b)^{1-D/2} L^{1/2}(n/b) \nonumber \\
&=& c\, \frac{b}{n} \kappa_1^{1/2} n^{1-D/2} L^{1/2}(n) b^{D/2-1} \frac{L^{1/2}(n/b)}{L^{1/2}(n) } \nonumber\\
&=& c\, \frac{d_n}{n} b^{D/2}  \left(\frac{L(n/b)}{L(n)}\right)^{1/2}\nonumber \\
&\sim& c\, \frac{d_n}{n} b^{D/2},
\label{eq:hb}
\end{eqnarray}
where we have used the fact that $L(n)$ is a slowly varying function. 

For the CUSUM test, we can apply Corollary~\ref{cor:Dn} and  we obtain  under the local alternative $A_n^C$,
that
\[
  \frac{1}{|a_1|} D_n 
  \claw \sup_{0\leq \lambda \leq 1} 
  \{(B_H(\lambda) - \lambda B_H(1) ) +\frac{1}{|a_1|} c\phi_\tau(\lambda)\}.
\]

For the Wilcoxon change-point test, we apply Corollary~\ref{cor:wil-alt} with $c$ replaced by $c\, b^{D/2}$, in view of (\ref{eq:hb}). We thus obtain under the local alternative $A_n^W$,
\begin{eqnarray*}
  \frac{1}{|\int_\R J_1(x) dF(x)|} W_n &\claw& \sup_{0\leq \lambda \leq 1} 
  \left\{(B_H(\lambda)-\lambda B_H(1)) +\frac{ \int_\R f^2(x) dx}{|\int_\R J_1(x) dF(x)| } c\, b^{D/2}\phi_\tau(\lambda)\right\}\\
&=&  \sup_{0\leq \lambda \leq 1} 
  \left\{(B_H(\lambda)-\lambda B_H(1)) +\frac{1}{|a_1|}  c\, \phi_\tau(\lambda)\right\},
\end{eqnarray*}
by (\ref{eq:hb}).
Let $q_\alpha$ denote the upper $\alpha$-quantile of the distribution of $ \sup_{0\leq \lambda \leq 1} 
  \{B_H(\lambda)-\lambda B_H(1)\}$; then the test that rejects the null hypothesis when 
$\frac{1}{|a_1|}D_n \geq q_\alpha$ or when $\frac{1}{|\int_\R J_1(x) dF(x)|} W_n\geq q_\alpha$, respectively, have asymptotically level $\alpha$. The power of these tests at the alternatives $A_n^C$  and $A_n^W$, respectively, converges to 
\[
  P\left(\sup_{0\leq \lambda \leq 1} 
  \left\{(B_H(\lambda)-\lambda B_H(1)) +\frac{1}{|a_1|}  c\, \phi_\tau(\lambda)\right\} \geq q_\alpha \right).
\]
Note that this also holds for the power along any other sequence, such as $bn$.
Since the level shift at the alternative $A_n^C$ equals the level shift at the alternative $A_{bn}^W$, we have shown that the Wilcoxon change-point test requires $b\, n$ observations to yield the same performance as the CUSUM test with $n$ observations. Thus $ARE(T_W,T_C)=1/b$, proving the theorem.
 \hfill $\Box$

\section{ARE of the Wilcoxon Change-Point Test and the CUSUM Test for IID Data}

We have shown in Example~\ref{ex:are-gauss} 
that in the case of LRD data, the ARE of the Wilcoxon change-point test and the 
CUSUM test is 1 for Gaussian data. In this section, we will compare this surprising result with the case 
of i.i.d.\ Gaussian data. We will find that in this case, the ARE is $\frac{3}{\pi}$, i.e.\ the
Wilcoxon change-point test is less efficient than the CUSUM test. 

We consider i.i.d.\ observations $X_1, \ldots, X_n$ with $X_i \sim \mathcal{N}(0,1)$ and the $U$-statistic
$$
U_k = \sum_{i=1}^k \sum_{j=k+1}^n h(X_i, X_j).
$$
As kernel we will choose $h_C(x,y)=y-x$ and $h_W(x,y)=I_{\{x\leq y\}}-\frac{1}{2}$, in other words we consider
\begin{align*}
 U_{k}^{(C)}   &=   \sum_{i=1}^k \sum_{j=k+1}^n h_C(X_i, X_j)   =   \sum_{i=1}^k \sum_{j=k+1}^n (X_j-X_i), \\
 U_{k}^{(W)}   &=   \sum_{i=1}^k \sum_{j=k+1}^n h_W(X_i, X_j)   =   \sum_{i=1}^k \sum_{j=k+1}^n \left( I_{\{X_i \leq X_j\}} -\frac{1}{2} \right).
\end{align*}
Both kernels $h_C, h_W$ are antisymmetric, i.e.\ they satisfy $h(x,y)=-h(y,x)$, so in order to determine the limit behaviour of $U_{k}^{(C)}$ and $U_{k}^{(W)}$, we can apply the limit theorems of Cs\"org\H{o} and Horv\'ath (1988).

\begin{theorem}
\label{THM:LimitTheorem_UC_UW_under_null}
Let $X_1, \ldots, X_n$ be i.i.d.\ random variables with $X_i \sim \mathcal{N}(0,1)$. Under the null hypothesis of no change in the mean, one has
\begin{align}
 \label{EQ:LimitTheorem_UC_under_null}
 \sup_{0\leq \lambda \leq 1} \left| \frac{1}{n^{3/2}} U^{(C)}_{[\lambda n]} - \BB_{1,n}(\lambda) \right| = o_P(1)
\end{align}
and
\begin{align}
 \label{EQ:LimitTheorem_UW_under_null}
 \sup_{0\leq \lambda \leq 1} \left| \frac{1}{n^{3/2} \sqrt{\frac{1}{12}}} U^{(W)}_{[\lambda n]} - \BB_{2,n}(\lambda) \right| = o_P(1),
\end{align}
where $(\BB_{i,n}(\lambda))_{0\leq\lambda\leq 1}$, $i=1,2$, is a sequence of Brownian bridges with mean $E[\BB_{i,n}(\lambda)]=0$ and auto-covariance $E[\BB_{i,n}(s) \BB_{i,n}(t)]=\min(s,t)-st$.
\end{theorem}

\begin{proof}
 By Theorem 4.1 of Cs\"org\H{o} and Horv\'ath (1988), it holds under the null hypothesis $H$ that
\begin{align*}
\sup_{0\leq \lambda \leq 1} \left| \frac{1}{n^{3/2} \sigma} U_{[\lambda n]} - \BB_n(\lambda) \right| = o_P(1)
\end{align*}
where $(\BB_{n}(\lambda))_{0\leq\lambda\leq 1}$  is a sequence of Brownian bridges like $\BB_{1,n}$ and $\BB_{2,n}$ above and where $\sigma^2 = E[\tilde{h}^2(X_1)]$ with $\tilde{h}(t) = E[h(t, X_1)]$. The kernel $h$ has to fulfill $E[h^2(X_1, X_2)] < \infty$ which is the case for $h_C(x,y)=y-x$ and $h_W(x,y)=I_{\{x\leq y\}}-\frac{1}{2}$ and Gaussian $X_i$.
\end{proof}

\begin{theorem}
\label{THM:LimitTheorem_UC_UW_under_alt}
Let $X_1, \ldots, X_n$ be i.i.d.\ random variables with $X_i \sim \mathcal{N}(0,1)$. Under the sequence of alternatives $A_{\tau, h_n}(n)$ and with $h_n=\frac{1}{\sqrt{n}} c$, where $c$ is a constant, one has
\begin{align}
 \label{EQ:LimitTheorem_UC_under_alt}
 \left( \frac{1}{n^{3/2}} U^{(C)}_{[\lambda n]} \right)_{0\leq \lambda\leq 1}   \to   \left( \BB_1(\lambda) + c\phi_\tau(\lambda) \right)_{0\leq \lambda\leq 1}
\end{align}
and
\begin{align}
 \label{EQ:LimitTheorem_UW_under_alt}
 \left( \frac{1}{n^{3/2} \sqrt{\frac{1}{12}}} U^{(W)}_{[\lambda n]} \right)_{0\leq \lambda\leq 1}   \to   \left( \BB_2(\lambda)  + \frac{c}{2 \sqrt{\pi} \cdot \sqrt{\frac{1}{12}}} \phi_\tau(\lambda) \right)_{0\leq \lambda\leq 1}
\end{align}
in distribution, where $(\BB_i(\lambda))_{0\leq\lambda\leq 1}$ is a Brownian bridge, $i=1,2$. 
\end{theorem}

\begin{proof}
First, we prove \eqref{EQ:LimitTheorem_UC_under_alt}. Like for the case of LRD observations, we decompose the statistic, so that we obtain under the sequence of alternatives $A_{\tau, h_n}(n)$
$$
\frac{1}{n^{3/2}} U^{(C)}_{[\lambda n]}   =   \frac{1}{n^{3/2}} \sum_{i=1}^{[\lambda n]} \sum_{j=[\lambda n]+1}^n (\epsilon_j-\epsilon_i)  +  \frac{1}{n^{3/2}} \sum_{i=1}^{[\lambda n]} \sum_{j=[\lambda n]+1}^n (\mu_j-\mu_i).
$$
By Theorem~\ref{THM:LimitTheorem_UC_UW_under_null}, the first term on the right-hand side converges to a Brownian bridge $\BB(\lambda)$. For the second term we have like in the proof for LRD observations
$$
\frac{1}{n^{3/2}} \sum_{i=1}^{[\lambda n]} \sum_{j=[\lambda n]+1}^n (\mu_j-\mu_i)   \sim   \sqrt{n} h_n \phi_\tau(\lambda),
$$
and in order for the right-hand side to converge, we have to choose 
\begin{equation}
h_n=\frac{1}{\sqrt{n}} c.
\label{eq:hn-iid}
\end{equation}
Now let us prove \eqref{EQ:LimitTheorem_UW_under_alt}. Again like for LRD observations, we decompose the statistic into one term that converges like under the null hypothesis and one term which converges to a constant. Under the sequence of alternatives $A_{\tau, h_n}(n)$ and for the case $\lambda \leq \tau$, this decomposition is
\begin{equation}
\label{EQ:Decomposition_UW_lambda<tau}
\frac{1}{n^{3/2}} U^{(W)}_{[\lambda n]}   =  \frac{1}{n^{3/2}} \sum_{i=1}^{[\lambda n]} \sum_{j=[\lambda n]+1}^n \left( I_{\{\epsilon_i \leq \epsilon_j\}} -\frac{1}{2} \right) + \frac{1}{n^{3/2}} \sum_{i=1}^{[\lambda n]} \sum_{j=[\tau n]+1}^n I_{\{\epsilon_j < \epsilon_i \leq \epsilon_j + h_n\}}. 
\end{equation}
The first term converges uniformly to a Brownian Bridge, like under the null hypothesis. We will show that, if the observations $\epsilon_i = G(\xi_i)$ are i.i.d.\ with c.d.f. $F$ which has two bounded derivatives $F'=f$ and $F''$, the second term converges uniformly to $c \lambda (1-\tau) \int_\R f^2(x)\, dx$, which is $c\phi_\tau(\lambda) \int_\R f^2(x)\, dx$ for the case $\lambda \leq \tau$. In the case of standard normally distributed $G(\xi_i)$, i.e.\ for $F=\Phi$ and $f=\varphi$, this function is $c (2\sqrt{\pi})^{-1} \phi_\tau(\lambda)$. To this end, we consider the following sequence of Hoeffding decompositions for the sequence of kernels $h_n(x,y) = I_{\{y < x \leq y + h_n\}}$:
\begin{equation}
\label{EQ:Sequence_Hoeffding_decompositions}
h_n(x,y) = \vartheta_n + h_{1,n}(x) + h_{2,n}(y) + h_{3,n}(x,y)
\end{equation}
Let $X,Y\sim F$ be i.i.d.~random variables. Then we define
\begin{eqnarray}
\vartheta_n   :=   E[h_n(X,Y)] 
   &=&   P(Y\leq X \leq Y+h_n)\nonumber \\   
   &=&   \int_\R \left( \int_y^{y+h_n} f(x) \; dx \right) f(y) \; dy \nonumber \\
   &=&  \int_\R \left( F(y+h_n)-F(y) \right) f(y) \; dy \nonumber\\
   &=&   h_n \int_\R \frac{F(y+h_n)-F(y)}{h_n} f(y) \; dy \nonumber \\
   &\sim&   h_n \int_\R f^2(y) \; dy,
\label{eq:limtheta}
\end{eqnarray}
where in the last step we have used that $(F(y+h_n)-F(y))/h_n\to f(y)$ and the dominated convergence theorem. Moreover, we set
\begin{eqnarray}
h_{1,n}(x)   :=   E[h_n(x, Y)] - \vartheta_n 
   &=&   E[I_{\{Y < x \leq Y + h_n\}}]-\vartheta_n \nonumber \\
   &=&   F(x) - F(x-h_n) - \vartheta_n
\label{eq:limh1}
\end{eqnarray}
and
\begin{eqnarray*}
h_{2,n}(y)   :=   E[h_n(X, y)] - \vartheta_n 
   &=&   E[I_{\{y < X \leq y+h_n\}}] - \vartheta_n\\
   &=&   F(y+h_n) - F(y) - \vartheta_n
\end{eqnarray*}
and 
\begin{eqnarray*}
h_{3,n}(x,y)   &:=&   h_n(x,y) - h_{1,n}(x) - h_{2,n}(y) - \vartheta_n \\
   &=&   I_{\{y < x \leq y + h_n\}} - F(x) + F(x-h_n) + \vartheta_n - F(y+h_n) + F(y).
\end{eqnarray*}
We will now show that
\begin{equation}
 \label{EQ:Hoeffding_h1_h2_h3_to_0}
 \sup_{0 \leq \lambda \leq \tau} \frac{1}{n^{3/2}} \left | \sum_{i=1}^{[\lambda n]} \sum_{j=[\tau n]+1}^n \left( h_{1,n}(\epsilon_i) + h_{2,n}(\epsilon_j) + h_{3,n}(\epsilon_i, \epsilon_j) \right) \right| \to 0
\end{equation}
in probability, and from this it follows by the sequence of Hoeffding decompositions~\eqref{EQ:Sequence_Hoeffding_decompositions} that
\begin{equation*}
 \sup_{0 \leq \lambda \leq \tau} \frac{1}{n^{3/2}} \left | \sum_{i=1}^{[\lambda n]} \sum_{j=[\tau n]+1}^n \left( h_n(\epsilon_i, \epsilon_j) - \vartheta_n \right) \right| \to 0
\end{equation*}
i.e.\ that the second term in \eqref{EQ:Decomposition_UW_lambda<tau} converges uniformly to 
\begin{align*}
\lim_{n\rightarrow \infty} \frac{1}{n^{3/2}} \sum_{i=1}^{[\lambda n]} \sum_{j=[\tau n]+1}^n \theta_n=
\lim_{n\to\infty} \frac{1}{n^{3/2}} [\lambda n] (n-[\tau n])  \vartheta_n   &=   \lambda (1-\tau) c \int_\R f^2(x) \; dx
\end{align*}
by (\ref{eq:limtheta}) and (\ref{eq:hn-iid}).

We use the triangle inequality and show the uniform convergence to 0 for each of the three terms in \eqref{EQ:Hoeffding_h1_h2_h3_to_0} seperately. Since the parameter $\lambda$ occurs only in the floor function value $[\lambda n]$, the supremum is in fact a maximum, and the $h_{1,n}(\epsilon_i)$ are i.i.d.\ random variables, so we can use Kolmogorov's inequality. We obtain for the first term in \eqref{EQ:Hoeffding_h1_h2_h3_to_0}
\begin{equation} 
 \label{EQ:KolmogorovInequality_for_h1}
 P\left( \sup_{0 \leq \lambda \leq \tau} \frac{n-[\tau n]}{n^{3/2}} \left|\sum_{i=1}^{[\lambda n]} h_{1,n}(\epsilon_i) \right| > s \right)   \leq   \frac{1}{s^2} \frac{n^2 (1-\tau)^2}{n^3} \sum_{i=1}^{[\tau n]} \Var[h_{1,n}(\epsilon_i)].
\end{equation}
Now consider an independent copy $\epsilon$ of the $\epsilon_i$ and the Taylor expansion of $F$ around the value of $\epsilon$, 
$$
F(\epsilon+h_n)=F(\epsilon)+F'(\epsilon) h_n+ \frac{F''(\epsilon+\delta)}{2} h_n^2,
$$
where the last term is the Lagrange remainder and thus $\epsilon+\delta$ is between $\epsilon$ and 
$\epsilon+h_n$. Then we obtain
\begin{align*}
 \frac{1}{h_n^2} \Var[h_{1,n}(\epsilon)]   &=   \Var\left[ \frac{F(\epsilon) - F(\epsilon - h_n)}{h_n} \right] \\
   &=   \Var\left[ f(\epsilon) + F''(\epsilon+\delta) \frac{h_n}{2} \right] \\
   &=  \Var\left[ f(\epsilon)\right] + \Var\left[F''(\epsilon+\delta) \frac{h_n}{2} \right] \\
   &\qquad + 2 \left(E[ f(\epsilon) F''(\epsilon+\delta) h_n ] - E[f(\epsilon)] \, E[F''(\epsilon+\delta) h_n] \right),
\end{align*}
and since $f=F'$ and $F''$ are bounded by assumption, we get $\Var[h_{1,n}(\epsilon)] = O(h_n^2)$.
Since $h_n\rightarrow 0$,  the right-hand side of \eqref{EQ:KolmogorovInequality_for_h1} converges to 0 as $n$ increases. 

In the same manner, we obtain
\begin{equation} 
 \label{EQ:KolmogorovInequality_for_h2}
 P\left( \sup_{0 \leq \lambda \leq \tau} \frac{[\lambda n]}{n^{3/2}} \left|\sum_{j=[\tau n]+1}^{n} h_{2,n}(\epsilon_j) \right| > s \right)   \leq   \frac{1}{s^2} \frac{n^2 \lambda^2}{n^3} \sum_{j=1}^{n} \Var[h_{2,n}(\epsilon_j)] \to 0.
\end{equation}

Finally, we have to show that 
\begin{equation} 
 \label{EQ:Hoeffding_h3_to_0}
  \sup_{0 \leq \lambda \leq \tau} \frac{1}{n^{3/2}} \left | \sum_{i=1}^{[\lambda n]} \sum_{j=[\tau n]+1}^n h_{3,n}(\epsilon_i, \epsilon_j) \right| \to 0
\end{equation}
in probability. We set temporarily $l:=[\lambda n] $ and $T:=[\tau n]$ and obtain from Markov's inequality
\begin{align*}
 P\left( \max_{0 \leq l \leq T} \frac{1}{n^{3/2}} \left | \sum_{i=1}^{l} \sum_{j=T+1}^n h_{3,n}(\epsilon_i, \epsilon_j) \right| > s \right)   &\leq   \frac{1}{s^2} E\left[ \max_{0 \leq l \leq T} \frac{1}{n^{3/2}} \sum_{i=1}^{l} \sum_{j=T+1}^n h_{3,n}(\epsilon_i, \epsilon_j) \right]^2.
\end{align*}
Now for any collection of random variables $Y_1, \ldots, Y_k$, one has $E[\max \{Y_1^2, \ldots Y_k^2\}] \leq \sum_{i=1}^k EY_i^2$, so that
\begin{align*}
 \frac{1}{s^2} E\left[ \max_{0 \leq l \leq T} \frac{1}{n^{3/2}} \sum_{i=1}^{l} \sum_{j=T+1}^n 
h_{3,n}(\epsilon_i, \epsilon_j) \right]^2   &\leq   \frac{1}{s^2} \frac{1}{n^3} \sum_{l=1}^T E \left[ \sum_{i=1}^{l} \sum_{j=T+1}^n h_{3,n}(\epsilon_i, \epsilon_j) \right]^2 \\
    &=   \frac{1}{s^2} \frac{1}{n^3} \sum_{l=1}^T \sum_{i=1}^{l} \sum_{j=T+1}^n \Var\left[h_{3,n}(\epsilon_i, \epsilon_j) \right],
\end{align*}
where in the last step we have used that $h_{3,n}(\epsilon_i,\epsilon_j)$ are uncorrelated by Hoeffding's decomposition. Now for two i.i.d.\ random variables $\epsilon$, $\eta$, we have, like above with the Taylor expansion of $F$:
\begin{align*}
 \Var\left[h_{3,n}(\epsilon, \eta) \right]   &=   \Var\left[I_{\{ \eta < \epsilon \leq \eta+h_n \}} -F(\epsilon) + F(\epsilon-h_n) + \vartheta_n - F(\eta+h_n) + F(\eta) \right]  \\
     &=   \Var\left[I_{\{ \eta < \epsilon \leq \eta+h_n \}} - h_n \left(f(\epsilon) + O_P(h_n)\right) + h_n \left(f(\eta) + O_P(h_n)\right) \right]  \\
     &=   \Var\left[I_{\{ \eta < \epsilon \leq \eta+h_n \}}\right] + \Var\left[ h_n \left(f(\epsilon) + f(\eta) + O_P(h_n)\right)\right] \\
     &\qquad + 2\Cov\left[I_{\{ \eta < \epsilon \leq \eta+h_n \}}, h_n \left(f(\epsilon) + f(\eta) + O_P(h_n)\right)\right]  \\
     &\leq   (\vartheta_n - \vartheta_n^2) + h_n^2 O(1) + 2 \sqrt{(\vartheta_n - \vartheta_n^2) \cdot h_n^2 O(1)}  \\
     &=   O(h_n),
\end{align*}
using (\ref{eq:limtheta}).
We have just shown that
$$
 P\left( \max_{0 \leq l \leq T} \frac{1}{n^{3/2}} \left | \sum_{i=1}^{l} \sum_{j=T+1}^n h_{3,n}(\epsilon_i, \epsilon_j) \right| > s \right)   \leq   \frac{1}{s^2} O(h_n),
$$
which proves \eqref{EQ:Hoeffding_h3_to_0}. So we have proven \eqref{EQ:Decomposition_UW_lambda<tau} for the case $\lambda \leq \tau$. The proof for $\lambda > \tau$ is similar.
\end{proof}

Now the stage is set to calculate the ARE of the Wilcoxon test based on $U^{(W)}_{[\lambda n]}$ and the CUSUM test based on $U^{(C)}_{[\lambda n]}$, as defined in the section about the ARE in the LRD case. Let $q_\alpha$ denote the upper $\alpha$-quantile of the distribution of $\sup_{0\leq \lambda \leq 1} \BB(\lambda)$. By Theorem~\ref{THM:LimitTheorem_UC_UW_under_alt}, the power of the tests is given
respectively by
\begin{equation}
 \label{EQ:Power_CUSUM_CPtest_iid}
 P\left( \sup_{0\leq \lambda \leq 1} \left(\BB(\lambda) + c_C \phi_\tau(\lambda) \right) \geq q_\alpha \right)
\end{equation}
and
\begin{equation}
 \label{EQ:Power_Wilcoxon_CPtest_iid}
 P\left( \sup_{0\leq \lambda \leq 1} \left(\BB(\lambda) + c_W \frac{1}{\sigma \cdot 2\sqrt{\pi}} \phi_\tau(\lambda) \right) \geq q_\alpha \right)
\end{equation}
where $\sigma^2 = 1/12$ and we assume that 
$$
h_W(n)   =   \frac{c_W}{\sqrt{n}}, \qquad h_C(n)   =   \frac{c_C}{\sqrt{n}}.
$$
Thus if we want both tests to have identical power, we must ensure that $c_C = c_W/(\sigma \cdot 2\sqrt{\pi})$, in other words
$$
\frac{h_C(n)}{h_W(n)}   =   \frac{c_C}{c_W}   =   \frac{1}{\sigma \cdot 2\sqrt{\pi}}.
$$
Now we define, as in  the proof for LRD observations, the probability
$$
  \psi(t):= P\left( \sup_{0\leq \lambda\leq 1} \left(\BB(\lambda) + t \, \phi_\tau(\lambda)\right) 
 \geq q_\alpha \right),
$$
for whose generalized inverse $\psi^-$ holds
\begin{equation}
\label{EQ:beta_iid}
 P\left( \sup_{0\leq \lambda\leq 1} \left(\BB(\lambda) +\psi^{-}(\beta) \, \phi_\tau(\lambda)\right) 
 \geq q_\alpha \right) \geq \beta.
\end{equation}
Now, comparing \eqref{EQ:beta_iid} and \eqref{EQ:Power_CUSUM_CPtest_iid}, we conclude that we can detect a level shift of size $h$ at time $[n\tau]$ with the CUSUM test of (asymptotic) level $\alpha$ and power $\beta$ based on $n$ observations, if $h_C(n)=\frac{c_C}{\sqrt{n}}$ and where $c_C$ satisfies $c_C =\psi^{-}(\beta)$; hence we obtain that 
$h_C(n)$ has to satisfy
$$
 h_C(n)  =   \frac{1}{\sqrt{n}} \psi^{-}(\beta).
$$
In the same manner, we get for the Wilcoxon test the conditions $h_W(n)=\frac{c_W}{\sqrt{n}}$ and $c_W/(\sigma 2 \sqrt{\pi}) = \psi^{-}(\beta)$ and thus
$$
 h_W(n)  =   \frac{\sigma 2 \sqrt{\pi}}{\sqrt{n}} \psi^{-}(\beta).
$$
Solving these two equations for $n$ again and denoting the resulting numbers of observations by $n_C$ and $n_W$, respectively, we obtain
\begin{align*}
 n_C   &=   \left( \frac{1}{h_C} \psi^{-}(\beta) \right)^2 \\
 n_W   &=   \left( \frac{2\sigma\sqrt{\pi}}{h_W} \psi^{-}(\beta) \right)^2.
\end{align*}
To obtain $ARE(T_W,T_C)$, we equate $h_W$ and $h_C$. We then obtain the following theorem. 

\begin{theorem}
\label{THM:ARE_TC_TW_iid_Gaussian}
Let $X_1, \ldots, X_n$ be i.i.d.\ random variables with $X_i \sim \mathcal{N}(0,1)$. Then 
\begin{align}
 ARE(T_W,T_C) = \lim_{h\to 0} \frac{n_C}{n_W} = (2\sigma\sqrt{\pi})^{-2}  =  \frac{3}{\pi},
\end{align}
where $T_C$, $T_W$ denote the one-sided CUSUM-test, respectively the one-sided Wilcoxon test, for the test problem $(H, A_{\tau, h_n})$. 
\end{theorem}

\section{Simulation Results}

We have proven that for Gaussian data, the CUSUM test and the Wilcoxon change-point test show asymptotically the same performance, i.e. that their ARE is 1. For Pareto(3,1) distributed data, we obtain,
using (\ref{eq:are-formula}) and numerical integration, an ARE of approximately $(2.68)^{2/D}$. Now we will illustrate these findings by a simulation study.

\subsection{Gaussian data}

We consider realizations $\xi_1,\ldots,\xi_n$ of a fGn process with Hurst parameter $H=0.7$ ($D=0.6$), using the \texttt{fArma} package in \texttt{R}, and create observations 
$$
 X_i=
\begin{cases}
	G(\xi_i)   &   \text{for } i=1,\ldots, [n\lambda] \\
	G(\xi_i) +h   &   \text{for } i=[n\lambda]+1,\ldots, n
\end{cases},
$$
by applying a transformation $G$ which is (with respect to the standard normal measure) normalized and square-integrable: $E[G(\xi)]=0$, $E[G^2(\xi)]=1$ for $\xi \sim \mathcal{N}(0,1)$. As a first step, we choose $G(t)=t$ in order to obtain Gaussian observations $X_1, \ldots, X_n$ (later we will choose a function $G$ such that we obtain Pareto distributed data). In other words, we consider data which follow the local alternative 
$$
 A_{\lambda,h}:
 \begin{cases}
	 \mu = E[X_i]=0   &   \text{for } i=1,\ldots,[n\tau] \\
	 \mu = E[X_i]=h   &   \text{for } i=[n\tau]+1,\ldots, n,
\end{cases}
$$
as in \eqref{eq:local-alt}.  In contrast to the simulations by Dehling, Rooch and Taqqu (2013), we choose a sample size $n=2,000$ instead of $n=500$. 
We let both the break point $k=[\tau n]$ and the level shift $h:=\mu_{k+1}-\mu_k$ vary; specifically, we choose $k=100, 200, 600, 1000$ (which corresponds to $\tau = 0.05, 0.1, 0.3, 0,5$) and we let  $h=0.5,1,2$. For each of these situations, we will compare the power of the CUSUM test and the power of the Wilcoxon change-point test in the test problem $(H,A_{\lambda,h})$: We have repeated each simulation $10,000$ times and counted, how often the respective test (correctly) rejected the null hypothesis. 

Since our theoretical considerations yield an ARE of 1, we expect that both tests detect jumps equally well -- that means that both tests, set on the same level, detect jumps of the same height and at the same position in the same number of observations with the same relative frequency. And indeed, we can clearly see in Table~\ref{TAB:Power_Cusum_Wilcox_fgn_N2000} that the power of both tests approximately coincides at many points; differences can be spotted only when the break occurs early in the data.

\begin{table}[htb]
\begin{center}
\begin{tabular}{c|cccc}
         &  \multicolumn{4}{c}{relative jump position $\tau$} \\
         &  0.05   &  0.1    &  0.3    &  0.5  \\
\hline
$h$=0.5  &  0.074  &  0.153  &  0.767  &  0.874  \\
$h$=1    &  0.153  &  0.694  &  1.000  &  1.000  \\
$h$=2    &  0.828  &  1.000  &  1.000  &  1.000
\end{tabular}
\quad
\begin{tabular}{c|cccc}
         &  \multicolumn{4}{c}{relative jump position $\tau$} \\
         &  0.05   &  0.1    &  0.3    &  0.5  \\
\hline
$h$=0.5  &  0.072  &  0.143  &  0.765  &  0.876  \\
$h$=1    &  0.128  &  0.602  &  1.000  &  1.000  \\
$h$=2    &  0.321  &  1.000  &  1.000  &  1.000
\end{tabular}
\\[2mm]
\caption[Power of CUSUM and Wilcoxon change-point test, fgn, sample size $n=2000$]{Power of the CUSUM test (left) and of the Wilcoxon change-point test (right), for $n=2000$ observations of fGn with LRD parameter $H=0.7$, different break points $[\tau n]$ and different level shifts $h$. Both tests have asymptotically level $5\%$. The calculations are based on 10,000 simulation runs.}
\label{TAB:Power_Cusum_Wilcox_fgn_N2000}
\end{center}
\end{table}

\subsection{Heavy tailed data}

We consider again realizations $\xi_1,\ldots,\xi_n$ of a fGn process with Hurst parameter $H=0.7$ ($D=0.6$) and create observations 
$$
 X_i=
\begin{cases}
	G(\xi_i)   &   \text{for } i=1,\ldots, [n\lambda] \\
	G(\xi_i) +h   &   \text{for } i=[n\lambda]+1,\ldots, n
\end{cases},
$$
by applying the transformation 
\[
G(t)   =   \frac{1}{\sqrt{3/4}}\left( (\Phi(t))^{-1/3}-\frac{3}{2} \right).
\]
In this case, the first Hermite coefficient of $G$, obtained by numerical integration, equals $a_1 \approx -0.6784$. This transformation $G$ produces observations $X_i=G(\xi_i)$ which follow a standardized Pareto$(3,1)$ distribution with mean zero and variance $1$. The probability density function of  $X_i$ is given by 
\[
f(x)=   
\begin{cases}
   3\sqrt{\frac{3}{4}} \left( \sqrt{\frac{3}{4}} \, 
 x + \frac{3}{2} \right)^{-4}	&	\text{if } x\geq -\sqrt{\frac{1}{3}}\\
   0	&	\text{else}.
\end{cases}
\]
To the second sample of observations, $X_{[\tau n]+1}, \ldots, X_n$, we again add a constant $h$, but this time we choose
\begin{equation}
 h = h_n   =   c \frac{d_n}{n}  =   c n^{-D/2}
\label{eq:hn-sim}
\end{equation}
as in \eqref{eq:hn}. We let the break point $k=[\tau n]$ vary; here, we choose $\tau = 0.05, 0.1, 0.3, 0.5$. We let also the sample size vary; we will give details below. To these data, we have applied the CUSUM test and the Wilcoxon change-point test, and under $10,000$ simulation runs we counted how often the respective test (rightly) rejected the null hypothesis. 

Now our theoretical considerations, see (\ref{eq:are-formula}), predict for this situation
\[
ARE   =   \lim_{n\to\infty} \frac{n_C}{n_W}   =   \left(\frac{|a_1| \int_\R f^2(x) \; dx}{|\int_\R J_1(x) f(x) \; dx|}\right)^{2/D}   \approx   (2.67754)^{2/0.6}   \approx   26.655.
\]
This means that the CUSUM test needs approximately 26.66 times as many observations as the Wilcoxon test in order to detect the same jump on the same level with the same probability. In order to find this behaviour, we have analysed the power of the Wilcoxon test for $n_W=10, 50, 100, 200$ observations and the power 
of the CUSUM test for $n_C=266, 1332, 2666, 5330$ observations. 

In order to be able to compare the two tests, we need to have identical level shifts when applying the Wilcoxon test to a sample of size $n_W$ and the CUSUM test to a sample of size $n_C=26.655\, n_W$. This can be achieved by choosing the constants $c$ in (\ref{eq:hn-sim}) accordingly, namely  taking $c_C=2.67754 \, c_W$. In this way, we obtain
\[
  h_{n_C}^C=c_C n_C^{-D/2}=2.67754 c_W (26.655 n_W)^{-D/2}= c_W n_W^{-D/2}=h_{n_W}^W.
\]
We ran simulations for two different choices of $c_W$, namely $c_W=1$ and $c_W=2$; see Table~\ref{TAB:Power_Cusum_Wilcox_pareto31_diverseN_c1} and Table~\ref{TAB:Power_Cusum_Wilcox_pareto31_diverseN_c2} for the results.

Here, we have to face a problem which was already encountered by Dehling, Rooch and Taqqu (2013). For the heavy-tailed Pareto data, the convergence of the CUSUM test statistic towards its limit is so slow  that the asymptotic quantiles of the limit distribution are not appropriate as critical values to define the domain of rejection of the test: In finite sample situations, the observed level of the test is not 5\% -- as it should be when using the 5\%-quantile of the asymptotic limit distribution. In order to remedy this, we used as critical value for the test,  the finite sample 5\% quantiles of the distribution of the CUSUM test statistic under the null hypothesis, using a Monte Carlo simulation; see Table~6 in Dehling, Rooch and Taqqu (2013). 
Here, we have performed the same steps, but for sample sizes $n=n_C=266, 1332, 2666, 5330$. The results are given in Table~\ref{TAB:Empirical_Quantiles_Cusum_pareto31}. Note that this problem does not arise when using the Wilcoxon change-point test, since the Wilcoxon test is distribution free under the null hypothesis.
\begin{table}[htb]
\begin{center}
\begin{tabular}{c|ccccc}
$n$                     &  266   &  1332  &  2666  &  5330   &   $\infty$  \\
\hline
$q_{\text{emp}, 0.05}$  &   0.73  &   0.66  &   0.64  &   0.63  &   0.59 
\end{tabular}
\\[2mm]
\caption[Empirical quantiles of CUSUM test for Pareto(3,1) data]{5\%-quantiles of the finite sample distribution of the CUSUM test under the null hypothesis for Pareto(3,1)-transformed fGn with LRD parameter $H=0.7$ and different sample sizes $n=n_C$.}
\label{TAB:Empirical_Quantiles_Cusum_pareto31}
\end{center}
\end{table}

The simulation results are shown in Table~\ref{TAB:Power_Cusum_Wilcox_pareto31_diverseN_c1} (for $c_W=1$) and Table~\ref{TAB:Power_Cusum_Wilcox_pareto31_diverseN_c2} (for $c_W=2$). Indeed, for a fixed jump position $\tau$,  the power of the CUSUM test (for $n=n_C=266, 1332, 2666, 5330$ observations) and of the Wilcoxon test (for $n=n_W=10, 50, 100, 200$ observations) coincide. They are not fully equal, but we conjecture this is due to the small sample size  which conflicts with the asymptotic character of our results. But it becomes clear: The CUSUM test needs quite a number of observations more to detect the same jump on the same level with the same probability -- as predicted by our calculation around 25 times as many. 

\begin{table}[htb]
\begin{center}
\begin{tabular}{rc|ccccc}
          &  &\multicolumn{4}{c}{relative jump position $\tau$} \\
          $n$ & $h$ &  0.05  &  0.1    &  0.3    &  0.5  \\
\hline
 266   & 0.50 &0.049  &  0.049  &  0.066  &  0.088  \\
1332  & 0.31 &0.050  &  0.052  &  0.083  &  0.110  \\
2666  &  0.25&0.052  &  0.055  &  0.092  &  0.127  \\
5330  &  0.20&0.051  &  0.054  &  0.099  &  0.130
\end{tabular}
\quad
\begin{tabular}{rc|ccccc}
         & & \multicolumn{4}{c}{relative jump position $\tau$} \\
        $n$ & $h$ &  0.05  &  0.1    &  0.3    &  0.5  \\
\hline
10   &0.50&  0.036  &  0.025  &  0.033  &  0.079  \\
50   & 0.31& 0.049  &  0.051  &  0.093  &  0.120  \\
100  & 0.25& 0.050  &  0.053  &  0.102  &  0.134  \\
200  &  0.20 &0.051  &  0.055  &  0.103  &  0.134
\end{tabular}
\\[2mm]
\caption[Power of CUSUM and Wilcoxon change-point test, Pareto(3,1) data, sample size $n=2000$]{Power of the CUSUM test (left) and of the Wilcoxon change-point test (right), at different break points $[\tau n]$, different sample sizes $n$, and different jump heights $h$, for Pareto(3,1) distributed data. Both tests have asymptotically level $5\%$ (CUSUM test is performed with empirical quantiles). The calculations are based on 10,000 simulation runs.}
\label{TAB:Power_Cusum_Wilcox_pareto31_diverseN_c1}
\end{center}
\end{table}

\begin{table}[htb]
\begin{center}
\begin{tabular}{rc|ccccc}
          &  & \multicolumn{4}{c}{relative jump position $\tau$} \\
         $n$ &$h$&   0.05   &  0.1     &  0.3     &  0.5  \\
\hline
 266   & 1.00  & 0.049  &   0.054  &   0.162  &   0.259 \\
1332  &  0.62 & 0.052  &   0.062  &   0.236  &   0.345 \\
2666  &   0.50& 0.055  &   0.069  &   0.272  &   0.390 \\
5330  &   0.41 & 0.054  &   0.074  &   0.287  &   0.402
\end{tabular}
\quad
\begin{tabular}{rc|ccccc}
         &&  \multicolumn{4}{c}{relative jump position $\tau$} \\
         $n$ &$h$ &   0.05  &  0.1    &  0.3    &  0.5  \\
\hline
10   & 1.00 & 0.033  &   0.024  &   0.039  &   0.197 \\
50   &  0.62 &0.049  &   0.055  &   0.199  &   0.283 \\
100  &  0.50 &0.051  &   0.063  &   0.225  &   0.316 \\
200  &   0.41&0.054  &   0.066  &   0.242  &   0.338
\end{tabular}
\\[2mm]
\caption[Power of CUSUM and Wilcoxon change-point test, Pareto(3,1) data, sample size $n=2000$]{Power of the CUSUM test (left) and of the Wilcoxon change-point test (right) at different break points $[\tau n]$, different sample sizes $n$, and different jump heights $h$,  for Pareto(3,1) distributed data. Both tests have asymptotically level $5\%$ (CUSUM test is performed with empirical quantiles). The calculations are based on 10,000 simulation runs.}
\label{TAB:Power_Cusum_Wilcox_pareto31_diverseN_c2}
\end{center}
\end{table}

\end{document}